\documentclass{amsart}

\makeatletter
\@addtoreset{equation}{section}

\makeatother

\usepackage{amsfonts, amsmath, amssymb, amsthm}
\usepackage{url}
\usepackage{cleveref}
\usepackage{graphicx}
\usepackage{color}
\usepackage{amscd}
\usepackage[right]{lineno}
\usepackage{mathrsfs}
\usepackage{MnSymbol}
\usepackage{enumerate}
\usepackage[all]{xy}

\newtheorem{theorem}{Theorem}[section]
\newtheorem{lemma}{Lemma}[section]
\newtheorem{proposition}{Proposition}[section]
\newtheorem{corollary}{Corollary}[section]

\theoremstyle{remark}

\newtheorem{remark}{Remark}[section]

\newtheorem{claim}{Claim}

\crefname{claim}{Claim}{Claims}

\newcommand{\ext}{{\rm Ext}}

\newcommand{\teich}{\mathcal{T}}

\newcommand{\mcg}{\operatorname{MCG}}
\newcommand{\tmod}{\operatorname{Mod}}
\newcommand{\torelliG}{\mathcal{I}}

\newcommand{\ThursM}{\mu_{Th}}
\newcommand{\PThursM}{\hat{\mu}_{Th}}

\newcommand{\PSiegelThursM}{\mathscr{M}_{Th}}

\newcommand{\convgenus}{{\boldsymbol \xi}}

\newcommand{\ray}{{\boldsymbol r}}

\newcommand{\ml}{\mathcal{ML}}
\newcommand{\pml}{\mathcal{PML}}
\newcommand{\syp}{\mathop{Sp}}
\newcommand{\psyp}{\mathop{PSp}}
\newcommand{\gl}{\mathop{GL}}
\newcommand{\spl}{\mathop{Sp}}
\newcommand{\symm}{\mathop{\mathbf{Sym}}}

\newcommand{\mat}{\mathrm{Mat}}

\newcommand{\symrep}{\mathop{\boldsymbol{\varrho}}}
\newcommand{\Cayley}{\Phi_{C\!a\!y}}

\newcommand{\hypell}{\operatorname{\mathbf{h\ell}}}

\newcommand{\busemann}{\boldsymbol{b}}

\newcommand{\HMh}{h}

\newcommand{\horo}{\mathfrak{horo}}

\newcommand{\Dirichlet}{\mathscr{D}}
\newcommand{\horos}{\mathscr{H}}

\newcommand{\conservativepart}{\mathfrak{Cons}}
\newcommand{\dissipativepart}{\mathfrak{Diss}}

\newcommand{\PSmeasure}{\mathbb{P}_{PS}}

\begin{document}

\title[Hopf decomposition]{Hopf decomposition of the actions of subgroups of the mapping class group}
\author{Hideki Miyachi}

\date{\today}
\address{School of Mathematics and Physics,
College of Science and Engineering,
Kanazawa University,
Kakuma-machi, Kanazawa,
Ishikawa, 920-1192, Japan
}
\email{miyachi@se.kanazawa-u.ac.jp}
\thanks{This work is partially supported by JSPS KAKENHI Grant Numbers
25K00909,
26K21806,
23K22396 }

\begin{abstract}
We study the Hopf decomposition of subgroup actions of the Teichm\"uller modular group on the Thurston boundary with respect to the Thurston measure class. We identify the conservative part with the big horospherical limit set modulo null sets. For a basepoint with trivial stabilizer in the subgroup, the ideal boundary of the associated Dirichlet polyhedron is wandering, and its subgroup translates cover the dissipative part modulo null sets. The dissipative part also agrees with the set of Dirichlet points modulo null sets. The description using Dirichlet polyhedra relies on a separation theorem for extremal length: every level set of an extremal length ratio at distinct points of Teichm\"uller space has measure zero. Kaimanovich's Radon--Nikodym criterion characterizes the two parts by the divergence and convergence, respectively, of a series of extremal length ratios.

For the Torelli group of a closed surface of genus at least two, we use radial limits of the period map to prove that its conical limit set has measure zero. Together with the conservativity established by Choi, Gekhtman, Yang, and Zheng, our geometric characterization implies that its big horospherical limit set has full measure and that the ideal boundary of every Dirichlet polyhedron has measure zero.
\end{abstract}

\maketitle
\tableofcontents
\clearpage

\section{Introduction}
Boundary actions of discrete groups can be studied from both geometric and measure-theoretic perspectives. A basic question is how the geometry of group orbits in the ambient space is reflected in the analytic, geometric, and measure-theoretic properties of the boundary action. For subgroups of the mapping class group (the Teichm\"uller modular group), this question concerns the relationship between the geometry of orbits in Teichm\"uller space and the dynamics of the action on the Thurston boundary with respect to the Thurston measure class.

This question is motivated by classical analogies between hyperbolic spaces and Teichm\"uller space, summarized in \Cref{table:dictionary}. These analogies also underlie the author's work on the pluricomplex Green function \cite{Miyachi2019}, Poisson integral representations \cite{Miyachi2023,miyachi2025bddpluriharmonicfunctionII-ToAppear}, and radial limits of bounded pluriharmonic functions \cite{Miyachi2024Bounded}. Here we apply this perspective to the measurable dynamics of subgroup actions on the boundary.

\begin{table}[htbp]
\centering
\setlength{\tabcolsep}{4.5pt}
\small
\begin{tabular}{|l|c|c|}
\hline
  & Hyperbolic space $\mathbb{H}^n$ & $\teich_{g,m}$ \\
\hline\hline
Distance & Hyperbolic distance & Teichm\"uller--Kobayashi distance \\ \hline
Metric & Hyperbolic metric & Teichm\"uller--Kobayashi metric \\ \hline
Discrete groups & Kleinian groups & Subgroups of $\tmod(g,m)$ \\ \hline
Compactification & $\overline{\mathbb{H}^n}=\mathbb{H}^n\cup \mathbb{S}^{n-1}$  & Thurston compactification \\
\hline
Boundary & $S^{n-1}$ & $\pml_{g,m}$ \\
\hline
Boundary measure & harmonic (visual) measure & (normalized) Thurston measure \\
\hline
\end{tabular}

\medskip
\caption{Analogies motivating the present research}
\label{table:dictionary}
\end{table}

The measure-theoretic study of Kleinian limit sets has a long history, including the Ahlfors measure conjecture and its resolution through the work of Canary, Agol, and Calegari--Gabai \cite{Ahlfors1966,Canary1993,Agol2004,CalegariGabai2006}. More directly related to the present paper, Sullivan's and Tukia's descriptions of boundary actions of Kleinian groups relate the Hopf decomposition to the geometry of horospherical limit sets \cite{Sullivan1981,Tukia1997}.

For general nonsingular actions, Kaimanovich \cite{Kaimanovich2010} gave criteria for the Hopf decomposition in terms of Radon--Nikodym derivatives. Choi, Gekhtman, Yang, and Zheng \cite{choi2024confinedsubgroupsgroupscontracting} developed related results for groups with contracting elements. In particular, they identify the infinitely conservative part with a big horospherical limit set in their boundary framework.

In the setting of the Thurston boundary, McCarthy and Papadopoulos studied subgroup limit sets and domains of discontinuity \cite{McCarthyPapadopoulos1989}. They also constructed fundamental domains in Teichm\"uller space analogous to Dirichlet polyhedra \cite{McCarthyPapadopoulos1996}. Kent and Leininger \cite{KentLeininger2008} studied conical limit points in connection with convex cocompactness. Horospherical limit points in the Thurston boundary were studied in \cite{miyachi2025functiontheorydynamicsergodic} and in joint work with Ohshika \cite{miyachi2026limitsetsmappingclass}. In that joint work, we give a complete classification of limit points of the action of the mapping class group in terms of the geometry of measured laminations.

A related measure-classification problem was studied by Choi and Kim \cite{ChoiKim2025InvariantMeasures}, who classified invariant Radon measures concentrated on the recurrence locus in the space of measured laminations for non-elementary subgroups of the mapping class group. Their results concern invariant measures on the non-projectivized space, whereas the present paper studies the Hopf decomposition of the action on the projectivized space with respect to the Thurston measure class.

\subsection{Main results}

We give two descriptions of the Hopf decomposition: one in terms of a series of Radon--Nikodym derivatives, and the other in terms of the big horospherical limit set, Dirichlet points, and Dirichlet polyhedra.

Let $\teich_{g,m}$ denote the Teichm\"uller space of Riemann surfaces of analytically finite type $(g,m)$, where $2g-2+m>0$, and let $\pml_{g,m}$ denote the space of projective measured laminations on a fixed complete hyperbolic surface of genus $g$ with $m$ cusps.
We identify $\pml_{g,m}$ with the Thurston boundary of
$\teich_{g,m}$.
Let $\tmod(g,m)$ denote the Teichm\"uller modular group,
namely the image of the mapping class group in the group
of automorphisms of $\teich_{g,m}$.

We assume throughout that $\convgenus=3g-3+m\ge2$.
The complexity-one cases $(g,m)=(1,1)$ and $(0,4)$
reduce to the classical setting of Fuchsian groups:
the Teichm\"uller space is identified with the hyperbolic
plane, and the Teichm\"uller modular group with
$\operatorname{PSL}_2(\mathbb Z)$.
Under the corresponding identification of the boundaries,
the Thurston measure class agrees with the Lebesgue
measure class on the circle.
The analogous descriptions of the Hopf decomposition
in these cases follow from classical results for
Fuchsian groups \cite{Sullivan1981,Tukia1997}.

\subsubsection{Hopf decomposition}
For $x\in\teich_{g,m}$, let $\ext_x(\lambda)$ denote the extremal length of a measured lamination $\lambda$ on $x$ (cf. \S\ref{subsec:HMdiffExtremalLength}).
The extremal length unit ball in the space of measured laminations determines a probability measure $\PThursM^x$ on $\pml_{g,m}$ by projectivizing and normalizing the Thurston measure (see \eqref{eq:def_Th_measure}).
The measures $\PThursM^x$ are mutually absolutely continuous and define the Thurston measure class on $\pml_{g,m}$.
The action of $\tmod(g,m)$ is nonsingular with respect to this measure
class, and its Radon--Nikodym derivatives are given by
\[
\frac{d([\omega]_*\PThursM^x)}{d\PThursM^x}([\lambda])
=\left(
\frac{\ext_x(\lambda)}{\ext_{[\omega](x)}(\lambda)}
\right)^{\convgenus}
\qquad ([\omega]\in\tmod(g,m));
\]
see \eqref{eq:Radon-Nikodym2}.

This formula leads us to consider the sum of these derivatives over
a subgroup. For a nontrivial subgroup $G\le\tmod(g,m)$, define
\begin{align*}
\conservativepart(G)
&=\conservativepart_x(G)
=\left\{[\lambda]\in\pml_{g,m}\mathrel{}\middle|\mathrel{}
\sum_{[\omega]\in G}
\left(
\frac{\ext_x(\lambda)}
{\ext_{[\omega](x)}(\lambda)}
\right)^{\convgenus}
=\infty
\right\},\\
\dissipativepart(G)
&=\dissipativepart_x(G)
=\left\{[\lambda]\in\pml_{g,m}\mathrel{}\middle|\mathrel{}
\sum_{[\omega]\in G}
\left(
\frac{\ext_x(\lambda)}
{\ext_{[\omega](x)}(\lambda)}
\right)^{\convgenus}
<\infty
\right\}.
\end{align*}
These sets are independent of the choice of $x$
(see \Cref{lem:Cons-Diss-ElementaryProperty}).
Our first main result applies the Radon--Nikodym criterion to identify
these sets with the conservative and dissipative parts of the action.

\begin{theorem}[Hopf decomposition]
\label{thm:Hopf_decomposition_main}
Let $G\le\tmod(g,m)$ be a nontrivial subgroup.
For every $x\in\teich_{g,m}$, the decomposition
\[
\pml_{g,m}=\conservativepart(G)\sqcup\dissipativepart(G)
\]
agrees, modulo $\PThursM^x$-null sets, with the Hopf decomposition
of the nonsingular action of $G$ on $(\pml_{g,m},\PThursM^x)$.
More precisely,
\begin{enumerate}
\item the action of $G$ on $\conservativepart(G)$ is conservative;
\item the action of $G$ on $\dissipativepart(G)$ is dissipative.
\end{enumerate}
\end{theorem}

The proof of \Cref{thm:Hopf_decomposition_main} also yields the following description of the ergodic decomposition:

\begin{corollary}[Ergodic decomposition]
\label{coro:ergodic_decomposition}
Let $G\le\tmod(g,m)$ be nontrivial. Then
\[
\conservativepart(G) =\mathscr{C}, \quad
\dissipativepart(G) =\mathscr{D}\setminus A_0,
\]
modulo Thurston null sets,
where
$\mathscr{C}$ is the union of all the purely nonatomic ergodic components of the action, $\mathscr{D}$ is the union of all the purely atomic ergodic components, and $A_0=\{[\lambda]\in \pml_{g,m}\mid \operatorname{Stab}_G([\lambda])\ne \{\textrm{id}\}\}$.
% Essential freeness
% (\Cref{lem:Fix_null}) implies that both $\mathscr{D}_\infty$ and $A_0$
% are null. See also \cite[Corollary 16]{Kaimanovich2010}.
\end{corollary}
The sets $\mathscr{C}$ and $\mathscr{D}$ are called the \emph{continual part} and the \emph{discontinual part} of the action of $G$, respectively (cf. \cite[Definition 2]{Kaimanovich2010}).

\subsubsection{Geometric characterization}
Our second main result gives a geometric description of the two parts, parallel to the classical results of Sullivan and Tukia.
Sullivan \cite{Sullivan1981} identified the conservative part of the boundary action, with respect to spherical Lebesgue measure, with the horospherical limit set modulo null sets.
Tukia \cite[Theorem~1]{Tukia1997} extended this description to nonatomic conformal measures by using the big horospherical limit set.
The big horospherical limit set $\Lambda_H(G)$ consists
of the boundary points for which some horoball centered at that point
contains infinitely many points of a $G$-orbit.
For $x_0\in\teich_{g,m}$, let $\operatorname{Dir}_G(x_0)$ denote
the set of Dirichlet points: these are the boundary points whose
associated Busemann function attains its infimum on $G(x_0)$.
For uniquely ergodic $[\lambda]$, this is equivalent to the
attainment of the minimum of $\ext_{[\omega](x_0)}(\lambda)$ over
$[\omega]\in G$; see \eqref{eq:horo_definition4} (cf. \cite[Theorem 29]{Kaimanovich2010}).

\begin{theorem}[Geometric characterization of the Hopf decomposition]
\label{thm:main_classfication_limit_sets}
Let $G\le\tmod(g,m)$ be a nontrivial subgroup.
Then $\conservativepart(G)$ agrees with $\Lambda_H(G)$ modulo Thurston null sets. Namely,
\[
\conservativepart(G)=
\left\{[\lambda]\in \pml_{g,m}\mid \text{$\exists t>0$ s.t. ${}^\#\left\{[\omega]\in G\mid \frac{\ext_x(\lambda)}{\ext_{[\omega](x)}(\lambda)}>t\right\}=\infty$}
\right\}
\]
modulo Thurston null sets.
If $x_0\in\teich_{g,m}$ satisfies $\operatorname{Stab}_G(x_0)=\{\mathrm{id}\}$, then
\[
\dissipativepart(G)
=\operatorname{Dir}_G(x_0)
=\bigcup_{[\omega]\in G}[\omega]\bigl(\Dirichlet_G^\infty(x_0)\bigr)
\]
modulo Thurston null sets.
Moreover, $\Dirichlet_G^\infty(x_0)$ is wandering under the action of $G$.
\end{theorem}

Here $\Dirichlet_G^\infty(x_0)$ denotes the ideal boundary of the Dirichlet polyhedron of McCarthy and Papadopoulos \cite{McCarthyPapadopoulos1996} centered at $x_0$; see \S\ref{subsec:DirichletDomain_Ideal}. Throughout the paper, null sets on $\pml_{g,m}$ are understood with respect to the Thurston measure class, or equivalently with respect to any of the measures $\PThursM^x$.

We prove \Cref{thm:Hopf_decomposition_main} in \S\ref{subsec:ProofOfDecomposition} and identify the conservative part in \S\ref{subsec:Proof_of_ConservativePart_BigHorospherical}. The assertions about the dissipative part and the ideal boundary follow from \Cref{thm:Dissipative_part,lem:almost-unique}.

Together, \Cref{thm:Hopf_decomposition_main,thm:main_classfication_limit_sets} give the following characterization of conservative subgroup actions, analogous to Sullivan's characterization for Kleinian groups \cite[Corollary in Section IV]{Sullivan1981}.

\begin{corollary}[A characterization of conservative actions]
\label{coro:conservativity_Dirichletset}
Let $G$ be a nontrivial subgroup of the Teichm\"uller modular group, and let $x_0\in\teich_{g,m}$ satisfy $\operatorname{Stab}_G(x_0)=\{\mathrm{id}\}$. For every $x\in\teich_{g,m}$, the following five conditions are equivalent:
\begin{itemize}
\item[(a)] the action of $G$ on $(\pml_{g,m},\PThursM^{x})$ is conservative;
\item[(b)] $\Lambda_H(G)$ has full $\PThursM^{x}$-measure;
\item[(c)] $\dissipativepart(G)$ has zero $\PThursM^{x}$-measure;
\item[(d)] $\operatorname{Dir}_G(x_0)$ has zero $\PThursM^{x}$-measure;
\item[(e)] the ideal boundary $\Dirichlet^\infty_G(x_0)$ of the Dirichlet polyhedron (in the sense of McCarthy--Papadopoulos) centered at $x_0$ has zero $\PThursM^{x}$-measure.
\end{itemize}
\end{corollary}

\subsubsection{Dynamics of the Torelli group}
As an application of our results, in \S\ref{sec:dynamics_of_Torelli_group} we study the boundary dynamics of the Torelli group $\torelliG_g$ of a closed surface of genus $g\ge 2$. Throughout this paper, we write $\pml_g=\pml_{g,0}$ for simplicity. The Torelli group is an infinite normal subgroup of the mapping class group $\mcg(\Sigma_g)$.

The Torelli group has been studied extensively from topological and algebraic perspectives (e.g., \cite{Johnson1983, Johnson1985, Johnson1985b, Morita1993, Putman2007, Putman2009}). Its geometric and dynamical properties, however, are less well understood (cf. \cite{BaikShinTranchida2024}). Our aim in this application is to present a geometric and measure-theoretic approach to the study of the Torelli group.

The non-ergodicity of the Torelli group action on the Thurston boundary was established in \cite{Miyachi2024Bounded}.
Since the Torelli group is an infinite normal subgroup of the mapping
class group, it is confined. The conservativity of its boundary action therefore follows from
\cite[Theorem~1.6(1)]{choi2024confinedsubgroupsgroupscontracting}.
The Torelli group is torsion-free, so the stabilizer condition in the geometric characterization is satisfied at every basepoint.
It then follows from \Cref{coro:conservativity_Dirichletset} that the big horospherical limit set has full measure and that the ideal boundary of a Dirichlet polyhedron has measure zero.

The conical limit set can also be shown to have measure zero using the non-ergodicity of the Torelli boundary action and the Hopf--Tsuji--Sullivan characterization in \Cref{thm:Hopf-Tsuji-Sullivan}; see also Yang \cite{yang2025conformaldynamicsinfinitygroups} for closed surfaces. In \S\ref{sec:dynamics_of_Torelli_group}, we give a direct geometric explanation of this fact using the period map, which is constant on Torelli orbits and does not increase the Kobayashi distance.

For almost every boundary direction, the radial limit of the period map lies on the boundary of the Siegel upper half-space (\Cref{lem:null_siegel}). In contrast, if a ray returns within bounded distance of a Torelli orbit at arbitrarily large times, its period matrices lie in a fixed compact subset of the Siegel upper half-space along the corresponding sequence of times. If the radial limit of the period map along this ray exists, it must therefore lie in the interior (\Cref{lem:conical_limit_set_torelli_group}). These two behaviors are incompatible, so the conical limit set has measure zero (\Cref{cor:conical_limit_sets}). Together with the preceding results, this gives the following corollary.

\begin{corollary}[Dynamics of the Torelli group]
\label{coro:torelli_main}
For $g\ge2$, the action of $\torelliG_g$ on $\pml_g$
is conservative but not ergodic with respect to the
Thurston measure class, and its big horospherical
limit set has full measure.
Moreover, the following hold:
\begin{itemize}
\item[(a)]
the conical limit set of the Torelli group has measure zero;
\item[(b)]
for every basepoint, the ideal boundary of the associated
Dirichlet polyhedron of the Torelli group has measure zero.
\end{itemize}
\end{corollary}
For the reader's convenience, we prove the Hopf--Tsuji--Sullivan characterization in the appendix for all surface types satisfying our standing assumptions (\Cref{sec:Hopf-Tsuji-Sullivan}).

For comparison, the action of the full Teichm\"uller modular group $\tmod(g)=\tmod(g,0)$ on $\pml_g$ is ergodic with respect to the Thurston measure class \cite{Masur1982,Rees1981}. Since the Thurston measure is nonatomic, ergodicity implies that the dissipative part has measure zero and hence that the action is conservative (see, for example, \cite[Theorem~14(ii)]{Kaimanovich2010}).

\Cref{table:comparison} compares the boundary dynamics of $\operatorname{PSL}_2(\mathbb{Z})$, the Teichm\"uller modular group, and the Torelli group.

\begin{table}[htbp]
\centering
\begin{tabular}{|c|c|c|c|}
\hline
& $\operatorname{PSL}_2(\mathbb{Z})$
& $\tmod(g)$
& $\mathcal{I}_g$ \\ \hline\hline
Limit set
& $\partial \mathbb{H}$
& \multicolumn{2}{c|}{$\pml_g$} \\
\hline
Ergodicity
& \multicolumn{2}{c|}{Yes}
& No \\ \hline
Conservativity
& \multicolumn{3}{c|}{Yes}
 \\ \hline
Measure of $\Lambda_C$
& \multicolumn{2}{c|}{$1$}
& $0$ \\ \hline
Measure of $\Lambda_H$
& \multicolumn{3}{c|}{$1$}\\ \hline
\end{tabular}
\caption{A dynamical comparison for $g\ge2$. For $\operatorname{PSL}_2(\mathbb Z)$,
measure refers to normalized visual measure on $\partial\mathbb H$;
for the other two groups, it refers to $\PThursM^{x_0}$ on $\pml_g$.
Ergodicity and conservativity are understood with respect to these
measure classes.}
\label{table:comparison}
\end{table}

Thus the Torelli group provides an example of a conservative boundary action whose conical limit set has measure zero and whose big horospherical limit set has full measure.

\subsubsection{Outline of the proofs of the main theorems}
We now outline the proofs of \Cref{thm:Hopf_decomposition_main,thm:main_classfication_limit_sets}. The geometric description of the dissipative part rests on the following separation property of extremal length. If $x_1\neq x_2$ and $c\in\mathbb R$, then

\[
\bigl\{[\lambda]\in\pml_{g,m}:
\ext_{x_1}(\lambda)
=e^c\ext_{x_2}(\lambda)\bigr\}
\]
has Thurston measure zero (\Cref{thm:level_set}). Since a subgroup orbit is countable, it follows that, for almost every $[\lambda]\in\pml_{g,m}$, the function $y\mapsto\ext_y(\lambda)$ takes distinct values at distinct points of any fixed subgroup orbit. Thus, whenever its minimum over that orbit is attained, the minimizing orbit point is unique. This property is independent of the representative $\lambda$, since extremal length is homogeneous of degree two.

When the basepoint has trivial stabilizer, this uniqueness implies that distinct translates of the ideal boundary of the Dirichlet polyhedron intersect only in null sets. We then identify their union with the Dirichlet set and with the dissipative part, modulo null sets (\Cref{thm:Dissipative_part}).

The proof of \Cref{thm:level_set} has two ingredients. First, near a measured lamination of rational depth zero, the Hubbard--Masur maps at two prescribed points are real analytic in a common train-track chart. Second, extremal length functions at distinct points cannot be proportional on a nonempty open set. The latter assertion follows from the differential formula for extremal length, the coordinates of \cite{MiyachiOhshika2017}, and the uniqueness of the Teichm\"uller geodesic joining two distinct points. A real-analytic function that is not identically zero on a connected open set has a zero set of Lebesgue measure zero. Since Thurston measure belongs to the Lebesgue measure class in train-track coordinates, these facts imply the separation theorem.

The Radon--Nikodym description of the Hopf decomposition is obtained by applying Kaimanovich's criterion \cite[Theorem~29]{Kaimanovich2010} to the essentially free boundary action. The explicit Radon--Nikodym formula converts that criterion into the extremal length series in \Cref{thm:Hopf_decomposition_main}. On the uniquely ergodic locus, the expression of Busemann functions in terms of extremal length gives the identification of the conservative part with the big horospherical limit set.

In \cite[Proposition~8.2]{miyachi2025functiontheorydynamicsergodic}, the author showed that the big horospherical limit set contains no wandering set of positive measure. The present results identify both parts of the Hopf decomposition and provide a geometric fundamental domain for its dissipative part. In particular, \Cref{coro:conservativity_Dirichletset} answers \cite[Problem~9]{miyachi2025functiontheorydynamicsergodic} when the boundary fundamental domain is interpreted as the ideal boundary of a Dirichlet polyhedron centered at a basepoint with trivial stabilizer.

\subsection{Organization of the paper}
The paper is organized as follows. In \S\ref{sec:preliminaries}, we review the geometric and measure-theoretic background. In \S\ref{sec:subgroup_dynamics}, we introduce the relevant limit sets and Dirichlet points. We prove the extremal length separation theorem in \S\ref{sec:extremal_separation} and use it in \S\ref{sec:Hopf_decomposition} to establish the Hopf decomposition and its geometric description. In \S\ref{sec:dynamics_of_Torelli_group}, we apply these results to the Torelli group. Finally, the appendix contains a proof of the Hopf--Tsuji--Sullivan characterization.

\subsection*{Acknowledgements}
The author is deeply grateful to Professor Athanase Papadopoulos and Professor Ken'ichi Ohshika for their continued encouragement and generous support.
The results of this paper were announced in a talk given by the author at the conference ``New Developments of Teichm\"uller Theory'', held at the Mathematisches Forschungsinstitut Oberwolfach in October 2026. The author thanks the organizers, participants, and staff of the institute for stimulating discussions and their hospitality.

\section{Notation and preliminaries}
\label{sec:preliminaries}
Throughout, we fix nonnegative integers $g$ and $m$ such that $2g-2+m>0$ and $\convgenus=3g-3+m\ge2$.
Let $\Sigma_{g,m}$ be an oriented surface of genus $g$ with $m$ punctures.
When $m=0$, we write $\Sigma_g$ for $\Sigma_{g,0}$.

\subsection{Teichm\"uller space}
A \emph{marked Riemann surface} $(M,f)$ of type $(g,m)$ is a pair consisting of an analytically finite Riemann surface $M$ of type $(g,m)$ and an orientation-preserving homeomorphism $f\colon \Sigma_{g,m}\to M$.
Two marked Riemann surfaces $(M_1,f_1)$ and $(M_2,f_2)$ are said to be \emph{Teichm\"uller equivalent} if there exists a biholomorphism $h\colon M_1\to M_2$ such that $h\circ f_1$ is homotopic to $f_2$ rel punctures.
By the uniformization theorem, any $x=(M,f)$ can be regarded as a marked complete hyperbolic surface of finite area by equipping $M$ with the complete hyperbolic metric in its conformal class.

\subsubsection{Topology of $\teich_{g,m}$}
The \emph{Teichm\"uller space} $\teich_{g,m}$ of Riemann surfaces of analytically finite type $(g,m)$ is the set of Teichm\"uller equivalence classes of marked Riemann surfaces of type $(g,m)$.
The Teichm\"uller space admits a canonical distance, called the \emph{Teichm\"uller distance}, defined by
\[
d_T(x_1,x_2)=\frac{1}{2}\log \inf_h K(h)
\]
for $x_i=(M_i,f_i)\in \teich_{g,m}$ ($i=1,2$), where $h$ runs over all quasiconformal mappings $h\colon M_1\to M_2$ such that $h\circ f_1$ is homotopic to $f_2$ rel punctures and $K(h)$ is the maximal dilatation of $h$.
The metric space $(\teich_{g,m},d_T)$ is complete and homeomorphic to Euclidean space of dimension $2\convgenus$.

\subsection{Measured laminations}
\label{subsec:measured_laminations}
Let $\mathcal{S}$ be the set of homotopy classes of 
essential simple closed curves on $\Sigma_{g,m}$.
Let $i(\alpha,\beta)$ denote the \emph{(geometric) intersection number} of simple closed curves $\alpha,\beta\in \mathcal{S}$.
For $t,s\ge0$ and $\alpha,\beta\in\mathcal S$, set
\begin{equation}
\label{eq:intersection-number-WS}
i(t\alpha,s\beta)=ts\,i(\alpha,\beta).
\end{equation}

\subsubsection{Measured laminations}
Henceforth, we fix a complete hyperbolic structure of finite area on $\Sigma_{g,m}$.
A \emph{geodesic lamination} $L$ on $\Sigma_{g,m}$ is a nonempty closed subset that is a disjoint union of complete simple geodesics, where a geodesic is said to be \emph{complete} if it is either closed or has infinite length in both directions.
The geodesics in $L$ are called the \emph{leaves} of $L$.
A \emph{transverse measure} on a geodesic lamination $L$ is an assignment of a Borel measure to each arc transverse to $L$, subject to the following two conditions: if a transverse arc $k'$ is contained in a transverse arc $k$, the measure assigned to $k'$ is the restriction of the measure assigned to $k$; and if two transverse arcs $k$ and $k'$ are homotopic through transverse arcs, the homotopy carries the measure assigned to $k$ to that assigned to $k'$.
A transverse measure on a geodesic lamination $L$ is said to have \emph{full support} if the support of the measure assigned to each transverse arc $k$ is exactly $k\cap L$.

A \emph{measured lamination} $\lambda$ is a pair consisting of a compact geodesic lamination $|\lambda|$, called the \emph{support} of $\lambda$, and a transverse measure of full support on $|\lambda|$.
For simplicity, we use the same symbol $\lambda$ for the measured lamination and its transverse measure.

Let $\ml$ be the set of measured laminations on $\Sigma_{g,m}$, together with the zero measured lamination $0$.
A weighted simple closed curve $t\alpha$ is identified with
a measured lamination supported on the simple closed geodesic homotopic to $\alpha$, with transverse measure equal to $t$ times the counting measure on its intersections with transverse arcs.

A measured lamination $\lambda$ is said to be \emph{essentially complete} if its support is maximal among geodesic laminations. Equivalently, $\lambda$ is essentially complete if and only if every component of $\Sigma_{g,m}\setminus |\lambda|$ is either an ideal triangle or a punctured monogon (cf. \cite[\S9.5]{Thurston-LectureNote}).

\subsubsection{Topology on $\ml$}
Birman and Series \cite{BirmanSeries1985} showed that the union of all complete simple geodesics on $\Sigma_{g,m}$ has Hausdorff dimension $1$. Consequently, one may choose geodesic arcs transverse to every complete simple geodesic; we call such arcs \emph{generic}. The topology on $\ml$ is defined as follows: a sequence $\{\lambda_n\}_n\subset \ml$ converges to $\lambda\in \ml$ if, for every generic transverse arc $k$ and every continuous function $\varphi$ on $k$,
\[
\int_k \varphi\,d\lambda_n\to \int_k\varphi\,d\lambda
\]
as $n\to \infty$ (cf. \cite[Part II]{Bonahon2001}). The intersection number \eqref{eq:intersection-number-WS} on weighted simple closed curves extends continuously to $\ml\times \ml$. Indeed, for $\lambda\in \ml$ and $\alpha\in \mathcal{S}$,
\[
i(\lambda,\alpha)=\inf_{\alpha'\in \alpha}\int_{\alpha'}d\lambda
\]
(e.g. \cite{Bonahon2001} and \cite{Thurston-LectureNote}).

The group $\mathbb R_{>0}$ acts on $\ml$ by scaling the transverse measure. The quotient space $\pml=(\ml-\{0\})/\mathbb{R}_{>0}$ is called the \emph{space of projective measured laminations}. Thurston showed that $\ml$ and $\pml$ are homeomorphic to $\mathbb{R}^{2\convgenus}$ and $\mathbb{S}^{2\convgenus-1}$, respectively.

The intersection number defines an embedding of $\ml$ into the space of nonnegative functions on $\mathcal{S}$:
\begin{equation}
\label{eq:ml-emb}
\ml\ni \lambda\mapsto (\mathcal{S}\ni \alpha\mapsto i(\lambda,\alpha))\in \mathbb{R}_{\ge 0}^{\mathcal{S}}.
\end{equation}

\subsubsection{Train track coordinates}
We follow Penner and Harer \cite{PennerHarer1992} for terminology and conventions concerning train tracks. We recall the notions used below.

A train track on $\Sigma_{g,m}$ is an embedded $1$-complex $\tau$ satisfying the following properties. Each edge (called a \emph{branch}) is a smooth path with well-defined tangent vectors at the endpoints, and at any vertex (called a \emph{switch}) the incident edges are mutually tangent. At each switch, the tangent vectors pointing into the incident branches have two possible directions. These directions divide the incident branch ends into two nonempty sets, which we call ``incoming'' and ``outgoing.'' The valence of each switch is at least $3$, except possibly for one bivalent switch in a closed curve component. Finally, each component $R$ of $\Sigma_{g,m}\setminus \tau$ must have negative generalized Euler characteristic. This is defined as the usual Euler characteristic $\chi(R)$ minus $1/2$ for each outward-pointing cusp and plus $1/2$ for each inward-pointing cusp (cf. \cite[\S1.1]{PennerHarer1992}).

For a train track $\tau$, a \emph{weight} on $\tau$ is a function $\mu$ on its set of branches that satisfies the following \emph{switch condition} at every switch $v$:
\[
\mu(b_1)+\cdots+\mu(b_r)=\mu(b_{r+1})+\cdots+\mu(b_{r+t}),
\]
where $b_1,\ldots,b_r$ are the incoming branch ends incident at $v$ and $b_{r+1},\ldots,b_{r+t}$ are the outgoing branch ends. We denote by $W(\tau)$ the vector space of weights on $\tau$ and by $E(\tau)\subset W(\tau)$ the cone of nonnegative weights on $\tau$. An element of $E(\tau)$ is also called a \emph{transverse measure} on $\tau$. A pair $(\tau,\mu)$ consisting of a train track $\tau$ and a transverse measure $\mu$ on $\tau$ is called a \emph{measured train track}. A measured train track $(\tau,\mu)$ is said to be \emph{positive} if $\mu(b)>0$ for every branch $b$ of $\tau$.

A train track $\tau$ is called \emph{recurrent} if it admits a positive transverse measure (cf. \cite[\S1.3]{PennerHarer1992}). Let $b(\tau)$ and $s(\tau)$ denote the numbers of branches and switches of $\tau$, respectively, and let $o(\tau)$ denote the number of orientable connected components of $\tau$. Then $\dim W(\tau)=b(\tau)-s(\tau)+o(\tau)$ (cf. \cite[Lemma 2.1.1]{PennerHarer1992}). In particular, if $\tau$ is connected, then $\dim W(\tau)=b(\tau)-s(\tau)$ when $\tau$ is non-orientable, whereas $\dim W(\tau)=b(\tau)-s(\tau)+1$ when $\tau$ is orientable.

We say that a measured lamination $\lambda$ on $\Sigma_{g,m}$ is \emph{carried} by a train track $\tau$ if there is a $C^1$ map $\phi\colon \Sigma_{g,m}\to \Sigma_{g,m}$, called a \emph{supporting map}, such that $\phi(|\lambda|)\subset \tau$, $\phi$ is homotopic to the identity, and, for every $p\in |\lambda|$, the restriction of the differential $d\phi$ to the tangent line of the leaf through $p$ is nonzero. For a train track $\tau$, we denote by $\ml(\tau)$ the set of measured laminations on $\Sigma_{g,m}$ that are carried by $\tau$.

A measured lamination $\lambda$ carried by $\tau$ determines a transverse measure $\mu=\mu_\lambda\in E(\tau)$ independently of the choice of supporting map (cf. \cite[Proposition 1.7.5]{PennerHarer1992}). Conversely, every transverse measure $\mu\in E(\tau)$ determines a measured lamination $\lambda_\mu$ on $\Sigma_{g,m}$, with the zero weight corresponding to the zero measured lamination (cf. \cite[Construction 1.7.7 and Corollary 2.7.3]{PennerHarer1992}).

When $\tau$ is recurrent, this correspondence gives a natural continuous bijection, called the \emph{realizing map} (\cite[Theorem 2.7.4]{PennerHarer1992}):
\begin{equation}
\label{eq:realizing_map}
R_\tau\colon E(\tau)\ni \mu\mapsto \lambda_\mu\in \ml(\tau)\subset \ml.
\end{equation}
The collection $\{(E(\tau),R_\tau)\mid \text{$\tau$ is recurrent}\}$ gives integral piecewise-linear coordinates on $\ml$ (cf. \cite[Proposition 9.5.8]{Thurston-LectureNote} and \cite[Theorem 3.1.4]{PennerHarer1992}).

\subsection{Hubbard--Masur differentials and extremal length}
\label{subsec:HMdiffExtremalLength}
Let $x=(M,f)\in\teich_{g,m}$. We denote by $\mathcal{Q}_x$ the finite-dimensional complex vector space of integrable holomorphic quadratic differentials on $M$, equipped with the $L^1$-norm
\[
\|q\|=\int_M |q(z)|\,\frac{|dz\wedge d\overline{z}|}{2},
\]
where $q=q(z)\,dz^2$ is the local expression for $q\in\mathcal{Q}_x$. Equivalently, after filling in the punctures of $M$, elements of $\mathcal{Q}_x$ are meromorphic quadratic differentials with at most simple poles at the punctures. For $q=q(z)\,dz^2\in\mathcal{Q}_x$, we define a measured lamination $v_x(q)$, called the \emph{vertical lamination} (or vertical foliation), by the condition
\[
i\bigl(v_x(q),\alpha\bigr)=\inf_{\alpha'\in\alpha}\int_{f(\alpha')}\left|\operatorname{Re}\bigl(\sqrt{q(z)}\,dz\bigr)\right|
\]
for every $\alpha\in\mathcal{S}$.

Hubbard and Masur \cite{HubbardMasur1979} showed that the map
\begin{equation}
\label{eq:HM-map}
v_x\colon \mathcal{Q}_x\longrightarrow \ml
\end{equation}
is a homeomorphism. We call \eqref{eq:HM-map} the \emph{Hubbard--Masur map}, although its definition predates their theorem. For each $\lambda\in\ml$, the \emph{Hubbard--Masur differential} $q_{\lambda,x}\in\mathcal{Q}_x$ is defined by $v_x(q_{\lambda,x})=\lambda$.

The \emph{extremal length} of $\lambda\in \ml$ on $x\in \teich_{g,m}$ is defined by
\[
\ext_x(\lambda)=\|q_{\lambda,x}\|.
\]
It satisfies $\ext_x(t\lambda)=t^2\ext_x(\lambda)$ for every $t\ge 0$. The extremal length function
\[
\teich_{g,m}\times \ml\ni (x,\lambda)\mapsto \ext_x(\lambda)
\]
is continuous and satisfies the following \emph{quasiconformal distortion estimate}:
\begin{equation}
\label{eq:QC-inv_extremal_length}
e^{-2d_T(x,y)}\le \frac{\ext_x(\lambda)}{\ext_y(\lambda)}\le e^{2d_T(x,y)}
\end{equation}
for all $x,y\in \teich_{g,m}$ and all $\lambda\in\ml\setminus\{0\}$.

\subsection{Teichm\"uller rays}
Let $x=(M,f)\in\teich_{g,m}$ and $[\lambda]\in\pml$. We define the \emph{Teichm\"uller (geodesic) ray} $\ray^\lambda_x\colon[0,\infty)\to\teich_{g,m}$ starting at $x$ in direction $[\lambda]$ by
\[
\ray^{\lambda}_x(t)=(M_t,f_t\circ f),
\]
where $\ray^{\lambda}_x(0)=x$ and $f_t\colon M\to M_t$ is the quasiconformal mapping with Beltrami coefficient $\tanh(t)\overline{q_{\lambda,x}}/|q_{\lambda,x}|$ for $t\ge 0$, interpreted as $0$ at the zeros of $q_{\lambda,x}$. This definition depends only on the projective class $[\lambda]$.
The ray $\ray^\lambda_x$ is a geodesic ray with respect to the Teichm\"uller distance:
\[
d_T(\ray^\lambda_x(t),\ray^\lambda_x(s))=|t-s|\quad (t,s\ge 0).
\]
Furthermore, for $x\in \teich_{g,m}$, the map
\[
\pml\times (0,\infty)\ni ([\lambda],t)\mapsto \ray^\lambda_x(t)\in \teich_{g,m}\setminus \{x\}
\]
is a homeomorphism.

\subsection{Thurston compactification}
For $x=(M,f)\in \teich_{g,m}$ and $\alpha\in \mathcal{S}$, we denote by $\hypell_x(\alpha)$ the length of the closed hyperbolic geodesic on $M$ homotopic to $f(\alpha)$.
The hyperbolic length function extends continuously and homogeneously to $\ml$; we denote this extension by $\hypell_x(\lambda)$.
The map 
\[
\teich_{g,m}\ni x\mapsto [\mathcal{S}\ni \alpha\mapsto \hypell_x(\alpha)]\in (\mathbb{R}^{\mathcal{S}}_{\ge 0}-\{0\})/\mathbb{R}_{>0}
\]
is an embedding, and the closure of its image is compact.
This is the \emph{Thurston compactification} of $\teich_{g,m}$.
Thurston showed that the compactification is homeomorphic to the closed ball of dimension $2\convgenus$, and the complement of the image of $\teich_{g,m}$ is naturally identified with $\pml$ via \eqref{eq:ml-emb}. 
The boundary $\pml$ of the Thurston compactification is called the \emph{Thurston boundary}.

\subsection{Action of the mapping class group on $\teich_{g,m}$ and $\pml$}

The mapping class group $\mcg(\Sigma_{g,m})$ consists of the
isotopy classes of orientation-preserving self-homeomorphisms of
$\Sigma_{g,m}$ that preserve the set of punctures. It acts on
$\teich_{g,m}$ by
\begin{equation}
\label{eq:action_of_mcg_Teichmuller}
[\omega](M,f)=(M,f\circ\omega^{-1}),
\end{equation}
where $\omega$ is a representative of
$[\omega]\in\mcg(\Sigma_{g,m})$.

Under the identification \eqref{eq:ml-emb}, the mapping class group
acts on $\ml$ by piecewise-linear homeomorphisms. More precisely,
\begin{equation}
\label{eq:action_of_mcg-ml}
i([\omega](\lambda),\alpha)
=
i(\lambda,[\omega]^{-1}(\alpha))
\end{equation}
for $[\omega]\in\mcg(\Sigma_{g,m})$, $\lambda\in\ml$, and
$\alpha\in\mathcal S$. Equivalently, $[\omega](\lambda)$ is the
measured lamination obtained by geodesically straightening
$\omega(|\lambda|)$ and transporting the transverse measure of
$\lambda$ by $\omega$. Since this action commutes with scalar
multiplication by $\mathbb R_{>0}$, it induces an action of
$\mcg(\Sigma_{g,m})$ on $\pml$.

The action of $\mcg(\Sigma_{g,m})$ on $\teich_{g,m}$ extends
continuously to the Thurston compactification. The induced action on
the Thurston boundary agrees with the action on $\pml$ described
above.

\subsubsection*{Teichm\"uller modular group}
For a few low-complexity surfaces, the actions of the mapping class group
$\mcg(\Sigma_{g,m})$ on $\teich_{g,m}$ and $\pml$ fail to be
faithful. Let
\begin{equation}
\label{eq:mcg_to_tmod}
\mcg(\Sigma_{g,m})
\longrightarrow
\operatorname{Aut}(\teich_{g,m})
\end{equation}
be the natural homomorphism defined by \eqref{eq:action_of_mcg_Teichmuller}.
Let $K_{g,m}$ be the kernel of \eqref{eq:mcg_to_tmod}.
Under the standing assumption $\convgenus\ge2$, $K_{g,m}$ is
also the kernel of the action on $\pml$. The quotient
\[
\tmod(g,m)
:=
\mcg(\Sigma_{g,m})/K_{g,m}
\]
is called the \emph{Teichm\"uller modular group} of type $(g,m)$, or simply the \emph{modular group} (cf.\ \cite[\S 2]{Bers1981} and \cite[\S 1]{BersBoundary1981}).
When $m=0$, we write $\tmod(g)$ for $\tmod(g,0)$.

The Teichm\"uller modular group $\tmod(g,m)$ acts faithfully on both $\teich_{g,m}$ and $\pml$, and its action on $\teich_{g,m}$ is properly discontinuous.
The term ``Teichm\"uller modular group'' is also used for the
mapping class group itself (cf. \cite{ImayoshiTaniguchi1992}, \cite{Ivanov1992} and \cite{Mumford1967}).

If $2g-2+m\ge3$, then $K_{g,m}$ is trivial, and hence
\[
\tmod(g,m)\cong\mcg(\Sigma_{g,m}).
\]
The exceptional cases are
\[
(g,m)=(0,4),\ (1,1),\ (1,2),\ \text{and }(2,0).
\]
More precisely,
\begin{align*}
K_{2,0}=\langle\iota\rangle,\ 
K_{1,2}=\langle\iota\rangle, \
K_{1,1}=\langle-I\rangle, \
K_{0,4}\cong
(\mathbb Z/2\mathbb Z)^2,
\end{align*}
where $\iota$ denotes the hyperelliptic involution. Consequently,
\begin{align*}
\tmod(2,0)
&\cong
\mcg(\Sigma_{2,0})/\langle\iota\rangle
\cong
\mcg(\Sigma_{0,6})
=
\tmod(0,6),\\
\tmod(1,2)
&\cong
\mcg(\Sigma_{1,2})/\langle\iota\rangle
\cong
\operatorname{Stab}_{\mcg(\Sigma_{0,5})}(p),\\
\tmod(1,1)
&\cong
\mcg(\Sigma_{1,1})/\langle-I\rangle
\cong
\operatorname{PSL}(2,\mathbb Z),\\
\tmod(0,4)
&\cong
\mcg(\Sigma_{0,4})/
(\mathbb Z/2\mathbb Z)^2
\cong
\operatorname{PSL}(2,\mathbb Z),
\end{align*}
where $p$ is the distinguished puncture of $\Sigma_{0,5}$ arising
from the quotient construction. The stabilizer of $p$ has index
five in $\mcg(\Sigma_{0,5})$. For the description of the kernels,
see \cite{BehrstockMargalit2006} and \cite{EarleKra1974}.

\subsection{Thurston measure}
A measured lamination $\lambda\in\ml$ is \emph{filling} if it satisfies $i(\lambda,\alpha)>0$ for every $\alpha\in\mathcal S$. A nonzero measured lamination $\lambda$ is called \emph{uniquely ergodic} if $i(\lambda,\mu)=0$ for $\mu\in \ml$ implies $\mu=t\lambda$ for some $t\ge 0$.
Let $\ml^{ue}$ be the set of uniquely ergodic measured laminations, and let $\pml^{ue}$ be its projectivization.

It is known that, up to multiplication by a positive constant, there is a unique nonzero locally finite ergodic $\mcg(\Sigma_{g,m})$-invariant Borel measure $\ThursM$ on $\ml$ that gives full measure to the set of filling measured laminations (cf.\ \cite{Masur1985} and \cite{LindenstraussMirzakhani2008}). We call $\ThursM$ the \emph{Thurston measure} on $\ml$.
The set $\ml^{ue}$ has full Thurston measure (cf.\ \cite{Masur1982}).
In train-track coordinates, the Thurston measure is in the Lebesgue measure class determined by the canonical Euclidean structure on the space of transverse measures (cf.\ \cite{MoninTelpukhovskiy2019}). The Thurston measure $\ThursM$ satisfies the homogeneity property
\begin{equation}
\label{eq:homogeneous_Thurston_measure}
\ThursM(tE)=t^{2\convgenus}\ThursM(E)
\end{equation}
for every measurable set $E\subset\ml$ and every $t>0$, where $tE=\{t\lambda\mid\lambda\in E\}$.

For $x\in\teich_{g,m}$, define
\[
\mathcal{BML}_x=\{\lambda\in\ml\mid\ext_x(\lambda)\le1\},
\]
and set $\mathcal{SML}_x=\partial\mathcal{BML}_x$.
The quotient map $\ml-\{0\}\ni \lambda\mapsto [\lambda]\in \pml$ induces a homeomorphism $\mathcal{SML}_x\to \pml$.
Under this identification, 
we define the \emph{Thurston measure} on $\pml$ with respect to $x\in \teich_{g,m}$ by
\begin{equation}
\label{eq:def_Th_measure}
\PThursM^x(A)=
\dfrac{\ThursM(\{\lambda\in\ml\setminus\{0\}\mid [\lambda]\in A, \ext_x(\lambda)\le 1\})}{\ThursM(\{\lambda\in\ml\mid \ext_x(\lambda)\le 1\})}
\end{equation}
for a Borel set $A\subset \pml$. By definition, each $\PThursM^x$ is a probability measure on $\pml$. For $x,y\in \teich_{g,m}$, $\PThursM^x$ and $\PThursM^y$ are mutually absolutely continuous. In fact, the Radon--Nikodym derivative satisfies
\begin{equation}
\label{eq:Radon-Nikodym}
\dfrac{d\PThursM^y}{d\PThursM^x}([\lambda])=
\left(\frac{\ext_x(\lambda)}{\ext_y(\lambda)}\right)^{\convgenus}
\end{equation}
for $[\lambda]\in \pml$ (cf. \cite[\S2.3.1]{Mirzakhani_etal2012}). Here we use the fact that the normalizing factor $\ThursM(\{\lambda\in\ml\mid\ext_x(\lambda)\le1\})$ is independent of $x$; this is the constancy of the Hubbard--Masur function. For analytically finite surfaces, including punctured surfaces, this follows from \cite[Corollary 15.2]{Miyachi2023}. For closed surfaces, see also \cite[Theorem 5.10]{Dumas2015}.

The following fact is well known; we include a proof for completeness.

\begin{lemma}
\label{lem:quasi-invariance}
Each $\PThursM^x$ is quasi-invariant under the action of the mapping class group.
\end{lemma}

\begin{proof}
Let $[\omega]\in\mcg(\Sigma_{g,m})$ and
$x\in\teich_{g,m}$.
The equivariance of extremal length gives
\[
\ext_{[\omega](x)}([\omega](\lambda))
=\ext_x(\lambda)
\]
for every $\lambda\in\ml$.
Together with the invariance of Thurston measure,
this implies, by \eqref{eq:def_Th_measure}, that
\[
\PThursM^{[\omega](x)}(A)
=\PThursM^x([\omega]^{-1}(A))
=([\omega]_*\PThursM^x)(A)
\]
for every Borel set $A\subset\pml$.
Since $\PThursM^{[\omega](x)}$ and $\PThursM^x$
are mutually absolutely continuous, the claim follows.
\end{proof}

The Radon--Nikodym derivative of the pushforward measure $[\omega]_*\PThursM^x$ with respect to $\PThursM^x$ is called the \emph{Jacobian} of $[\omega]$ with respect to $\PThursM^x$ (cf.\ \cite[Appendix D11]{HubbardVol3_2022}).
The preceding proof and \eqref{eq:Radon-Nikodym} give the following formula for this Jacobian:
\begin{equation}
\label{eq:Radon-Nikodym2}
\dfrac{d([\omega]_*\PThursM^x)}{d\PThursM^x}([\lambda])=
\dfrac{d\PThursM^{[\omega](x)}}{d\PThursM^x}([\lambda])=
\left(\frac{\ext_x(\lambda)}{\ext_{[\omega](x)}(\lambda)}\right)^{\convgenus}
\end{equation}
for $[\lambda]\in \pml$ and $[\omega]\in \mcg(\Sigma_{g,m})$.

The ergodicity of the mapping class group action on $\pml$ with respect to the Lebesgue measure class is due to Masur \cite{Masur1982}; see also Rees \cite{Rees1981} for an alternative proof that applies to surfaces with boundary.
The Thurston measure $\PThursM^x$ is nonatomic. Thus, for each $x\in\teich_{g,m}$, the completion of $(\pml,\PThursM^x)$ is a Lebesgue space in the sense of Rohlin, isomorphic modulo null sets to $[0,1]$ with Lebesgue measure (cf.\ \cite{Rohlin1952}; see also \cite[Theorem 17.41]{Kechris1995}).

\subsection{Busemann functions and horoballs}
The \emph{Busemann function} $\busemann_{\lambda}^{x_0}$ with direction $[\lambda]\in \pml$ and basepoint $x_0\in \teich_{g,m}$ is defined by the limit
\[
\busemann^{x_0}_{\lambda}(x)=\lim_{t\to \infty}(d_T(x,\ray^{\lambda}_{x_0}(t))-t)
\]
(cf. \cite[\S22]{Busemann1955} and \cite[\S3]{EberleinONeill1973}).
By definition, $\busemann^{x_0}_{\lambda}(x_0)=0$, and
\begin{equation}
\label{eq:busemann-basepoint}
\busemann^{[\omega](x_0)}_{\lambda}([\omega](x))
=\busemann^{x_0}_{[\omega]^{-1}(\lambda)}(x).
\end{equation}
Ivanov \cite{Ivanov2001} showed that, for any $x_1,x_2\in \teich_{g,m}$ and $[\lambda]\in\pml$, there exists $D_0>0$ such that $d_T(\ray^\lambda_{x_1}(t),\ray^\lambda_{x_2}(t))\le D_0$ for all $t\ge 0$. Consequently, changing the basepoint from $x_0$ to $x_1$ gives the estimate
\begin{equation}
\label{eq:Busemann_function}
|\busemann^{x_0}_{\lambda}(x)-\busemann^{x_1}_{\lambda}(y)|\le D_0+d_T(x,y)
\end{equation}
for all $x,y\in \teich_{g,m}$, where $D_0$ depends on $x_0,x_1$, and $[\lambda]$.
When $[\lambda]$ is uniquely ergodic, the Busemann function is given by
\begin{equation}
\label{eq:busemann_ue}
\busemann^{x_0}_\lambda(x)=\frac{1}{2}(\log \ext_{x}(\lambda)-\log \ext_{x_0}(\lambda))
\end{equation}
(cf. \cite{LiuSu2014} and \cite{Walsh2019}; see also \cite{Miyachi2017}). For $[\lambda]\in \pml^{ue}$, formula \eqref{eq:busemann_ue} gives the cocycle identity
\begin{equation}
\label{eq:busemann_cocycle_condition}
\busemann^{x_0}_\lambda(y_0)+\busemann^{y_0}_\lambda(z_0)+\busemann^{z_0}_\lambda(x_0)=0.
\end{equation}
For $r\in \mathbb{R}$, we define the open $r$-\emph{horoball} centered at $[\lambda]\in \pml$ with basepoint $x_0\in \teich_{g,m}$ by
\[
\horos([\lambda],r;x_0)=
\left\{x\in \teich_{g,m}\mid \busemann^{x_0}_\lambda(x)<r\right\}.
\]
Its closure is $\overline{\horos([\lambda],r;x_0)}=\{x\in \teich_{g,m}\mid \busemann^{x_0}_\lambda(x)\le r\}$.
The equivariance identity \eqref{eq:busemann-basepoint} gives
\begin{equation}
\label{eq:horoball-basepoint}
[\omega](\horos([\lambda],r;x_0))
=\horos([\omega]([\lambda]),r;[\omega](x_0)).
\end{equation}
For $[\lambda]\in\pml^{ue}$, the cocycle condition \eqref{eq:busemann_cocycle_condition} gives
\begin{equation}
\label{eq:horoball-basepoint2}
\horos([\lambda],r;y_0)
=\horos([\lambda],r+\busemann^{x_0}_\lambda(y_0);x_0).
\end{equation}
The boundary $\partial \horos([\lambda],r;x_0)$ is the level set $\{\busemann^{x_0}_\lambda=r\}$ of the Busemann function. We call it the $r$-\emph{horosphere} centered at $[\lambda]$ with basepoint $x_0$. When $r=0$, we write $\horos([\lambda];x_0)$ for the $0$-horoball with basepoint $x_0$. The point $x_0$ lies on the boundary of this horoball.

\begin{remark}
In \cite{SuTan2023}, W. Su and D. Tan discuss properties of horoballs and horospheres. In particular, they observe that formula \eqref{eq:busemann_ue} also holds for indecomposable measured laminations (cf. \cite[Proposition 5.7]{SuTan2023}).
\end{remark}

\section{Dynamics of subgroups}
\label{sec:subgroup_dynamics}
This section recalls the subgroup classification and the geometric boundary loci used in our description of the Hopf decomposition. Throughout this section, let $G$ be a subgroup of $\tmod(g,m)$.

\subsection{McCarthy--Papadopoulos classification}
The subgroup classification of McCarthy and Papadopoulos \cite{McCarthyPapadopoulos1989} is modeled on the theory of Kleinian groups. We recall their notion of a \emph{sufficiently large subgroup}, the mapping class group analogue of a non-elementary Kleinian group.

Although their classification is formulated for subgroups of $\mcg(\Sigma_{g,m})$, we work with subgroups of the Teichm\"uller modular group $\tmod(g,m)$. In other words, we pass to the effective quotient of the actions on $\teich_{g,m}$ and $\pml$, so that both actions are faithful.

\subsubsection{Limit sets}

The notions of limit point and limit set are due to McCarthy and Papadopoulos \cite{McCarthyPapadopoulos1989} and Masur \cite{Masur1986}; see also Kent and Leininger \cite[\S3]{KentLeininger2008}.

A \emph{weak limit point} of the action of $G$ on $\pml$ is a point $[\lambda]\in\pml$ for which there exist $[\lambda']\in\pml$ and a sequence of pairwise distinct elements $\{[\omega_n]\}_n\subset G$ such that $[\lambda]=\lim_{n\to\infty}[\omega_n]([\lambda'])$. The \emph{canonical limit set} is the closure of the set of all weak limit points. A \emph{limit set} of $G$ is any closed $G$-invariant subset of the canonical limit set.

We say that $G$ is \emph{dynamically irreducible} if it has a unique nonempty minimal limit set. We call this set the \emph{limit set} of $G$ and denote it by $\Lambda(G)$. We refer to the points of $\Lambda(G)$ as the \emph{limit points} of $G$.

\subsubsection{Sufficiently large subgroups}

A set of pseudo-Anosov mapping classes is said to be \emph{independent} if no two elements of the set have the same fixed-point set in $\pml$. A subgroup $G$ of $\tmod(g,m)$ is called \emph{sufficiently large}\index{sufficiently large} if it contains an independent pair of pseudo-Anosov elements.

Every sufficiently large subgroup is dynamically irreducible. Indeed, let $\Lambda_0(G)\subset\pml$ be the set of fixed points of pseudo-Anosov elements of $G$. If $G$ is sufficiently large, the closure of $\Lambda_0(G)$ is the unique nonempty minimal closed $G$-invariant subset of $\pml$. It therefore coincides with the limit set $\Lambda(G)$ (cf.\ \cite[Corollary 4.2]{McCarthyPapadopoulos1989}).

If $G$ is sufficiently large and $H<G$ is either an infinite normal subgroup or a finite-index subgroup, then $H$ is also sufficiently large and $\Lambda(H)=\Lambda(G)$ (cf.\ \cite[Propositions 5.1, 5.4, and 5.5]{McCarthyPapadopoulos1989}).

Let
\[
Z\Lambda(G)=\{[\lambda]\in\pml\mid
i(\lambda,\mu)=0\text{ for some }[\mu]\in\Lambda(G)\}.
\]
McCarthy and Papadopoulos show that if a sequence of pairwise distinct points in the orbit $G(p)$ converges to $[\lambda]\in\pml$, where $p\in\teich_{g,m}\cup\pml$, then $[\lambda]\in Z\Lambda(G)$. If, in addition, $[\lambda]\in\pml^{ue}$, then $[\lambda]\in\Lambda(G)$ (cf.\ \cite[Proposition 8.1]{McCarthyPapadopoulos1989}).

\subsubsection*{Examples}
The following examples are discussed in \cite[Example 1 in \S5]{McCarthyPapadopoulos1989}.
\begin{itemize}
\item[(a)]
The Teichm\"uller modular group $\tmod(g,m)$ is sufficiently large. Its action on $\pml$ is also minimal, and hence
\[
\Lambda(\tmod(g,m))=\pml
\]
(cf.\ \cite[Expos\'e 6, VII]{Fathi1979}).
\item[(b)]
For $g\ge 2$, the natural quotient homomorphism \eqref{eq:mcg_to_tmod} restricts to an embedding of the Torelli group $\torelliG_g$.
Indeed, the kernel of the quotient homomorphism is trivial when $g\ge 3$. When $g=2$, it is generated by the hyperelliptic involution $\iota$, which acts as $-I$ on $H_1(\Sigma_g,\mathbb Z)$ and therefore does not belong to $\torelliG_g$. We may thus identify $\torelliG_g$ with its image in $\tmod(g)$. Since this image is an infinite normal subgroup of $\tmod(g)$, it follows that $\torelliG_g$ is sufficiently large and
that
$\Lambda(\torelliG_g)=\pml$.
\end{itemize}

\subsection{Conical and horospherical limit points}

Following Kent and Leininger \cite{KentLeininger2008}, we say that a point $[\lambda]\in\pml$ is a \emph{conical limit point} of $G$ if, for every Teichm\"uller geodesic ray $\ray^\lambda_*$ in direction $[\lambda]$, there are $R>0$ and a $G$-orbit that meets the $R$-neighborhood of $\ray^\lambda_*$ in infinitely many points.
Kent and Leininger also require conical limit points to belong to $\Lambda(G)$. Our definition is equivalent to theirs; see \cite{miyachi2026limitsetsmappingclass}. By Masur's criterion, every conical limit point is uniquely ergodic \cite{Masur1992}. We denote the set of conical limit points of $G$ by $\Lambda_C(G)$.

We next recall the notions of small and big horospherical limit points, introduced by Ohshika and the author in \cite{miyachi2026limitsetsmappingclass}. A point $[\lambda]\in\pml$ is called a \emph{small horospherical limit point} of $G$ if there exist $x_0\in\teich_{g,m}$ and a sequence $([\omega_n])_{n=1}^\infty\subset G$ such that
\begin{equation}
\label{eq:small-horo_definition}
\busemann^{x_0}_\lambda([\omega_n](x_0))\longrightarrow-\infty
\end{equation}
as $n\to\infty$. 
A point $[\lambda]\in\pml$ is called a \emph{big horospherical limit point} of $G$ if there exist $x_0\in\teich_{g,m}$, $b\in \mathbb{R}$, and an infinite sequence $([\omega_n])_{n=1}^\infty\subset G$ of pairwise distinct elements such that
\begin{equation}
\label{eq:big_horo_definition}
\busemann^{x_0}_\lambda([\omega_n](x_0))<b
\end{equation}
for every $n\ge 1$. 

These notions are analogous to those for visibility manifolds (cf.\ \cite[\S3]{EberleinONeill1973} and \cite[Section IV]{Sullivan1981}). Geometrically, $[\lambda]\in\pml$ is a small horospherical limit point if and only if every horoball centered at $[\lambda]$ contains infinitely many points of the orbit $G(x_0)$. It is a big horospherical limit point if and only if some horoball centered at $[\lambda]$ contains infinitely many points of that orbit. Equation \eqref{eq:Busemann_function} shows that both definitions are independent of the choice of basepoint.

We denote the sets of small and big horospherical limit points of $G$ by $\Lambda_h(G)$ and $\Lambda_H(G)$, respectively. Then
\[
\Lambda_C(G)\subset\Lambda_h(G)\subset \Lambda_H(G)
\]
(cf.\ \cite{miyachi2025functiontheorydynamicsergodic}). Suppose that $[\lambda]\in\pml^{ue}$ is a small horospherical limit point of $G$. Choose a sequence $([\omega_n])_{n=1}^\infty\subset G$ such that
\begin{equation}
\frac{\ext_{[\omega_n](x_0)}(\lambda)}
{\ext_{x_0}(\lambda)}
\longrightarrow 0.
\end{equation}
Then $\{[\omega_n](x_0)\}_{n=1}^\infty$ converges to $[\lambda]$ in the Thurston compactification. Indeed,
\begin{equation*}
e^{-2d_T(x_0,[\omega_n](x_0))}
\ext_{[\omega_n](x_0)}(\lambda)
\longrightarrow 0
\end{equation*}
as $n\to\infty$. Together with the unique ergodicity of $\lambda$, this implies that $[\omega_n](x_0)$ converges to $[\lambda]$ in the Gardiner--Masur compactification (cf.\ \cite[Theorem 1.1]{Miyachi2008} and \cite[Theorem 3]{Miyachi2013}). Consequently, it also converges to $[\lambda]$ in the Thurston compactification (cf.\ \cite[Corollary 1]{Miyachi2013}).

\subsection{Dirichlet points}
We set
\begin{equation}
\label{eq:horo_definition3}
\horo^{x_0}_G([\lambda])
=
\inf_{[\omega]\in G}
\exp\left(2\busemann^{x_0}_\lambda([\omega](x_0))\right).
\end{equation}
By \eqref{eq:busemann_ue}, for $[\lambda]\in\pml^{ue}$, we have
\begin{equation}
\label{eq:horo_definition4}
\horo^{x_0}_G([\lambda])
=
\inf_{[\omega]\in G}
\frac{\ext_{[\omega](x_0)}(\lambda)}
{\ext_{x_0}(\lambda)}.
\end{equation}
A point $[\lambda]\in\pml^{ue}$ is a small horospherical limit point of $G$ if and only if
\begin{equation}
\label{eq:horo_definition2}
\horo^{x_0}_G([\lambda])=0.
\end{equation}
Following the terminology for Kleinian groups, we call $[\lambda]\in\pml$ a \emph{Dirichlet point} of $G$ with respect to $x_0$ if $\horo^{x_0}_G([\lambda])>0$ and the infimum defining $\horo^{x_0}_G([\lambda])$ is attained at a point of the orbit $G(x_0)$. We denote the set of such points by $\operatorname{Dir}_G(x_0)$ and call it the \emph{Dirichlet set} of $G$.

\section{Level sets of extremal length ratios}
\label{sec:extremal_separation}

\subsection{Definitions and the nullity theorem}

For $x_1,x_2\in\teich_{g,m}$ and $c\in\mathbb{R}$, define the \emph{extremal length level sets} by
\begin{align*}
\ml_c(x_1,x_2)
&=\{\lambda\in\ml\mid \ext_{x_1}(\lambda)=e^c\ext_{x_2}(\lambda)\},\\
\pml_c(x_1,x_2)
&=\{[\lambda]\in\pml\mid \ext_{x_1}(\lambda)=e^c\ext_{x_2}(\lambda)\}.
\end{align*}
The second definition is independent of the choice of representative $\lambda$ because extremal length is homogeneous of degree two.
When $c=0$, we simply write $\ml(x_1,x_2)=\ml_0(x_1,x_2)$ and $\pml(x_1,x_2)=\pml_0(x_1,x_2)$.
 
The following theorem shows that these level sets are null whenever the basepoints are distinct.

\begin{theorem}
\label{thm:level_set}
If $x_1\ne x_2$, then $\ml_c(x_1,x_2)$ is null with respect to $\ThursM$, and $\pml_c(x_1,x_2)$ is null with respect to $\PThursM^x$ for every $x\in\teich_{g,m}$.
\end{theorem}

The proof combines local real analyticity with a rigidity statement. \Cref{lem:HM_local_analytic,lem:common_splitting} provide common train-track coordinates in which the two extremal length functions are real analytic near almost every measured lamination. \Cref{lem:coincidence} rules out proportionality on an open set. These two facts reduce the nullity assertion to the zero-set theorem for real-analytic functions. The treatment of simple poles in \Cref{lem:HM_local_analytic} is included below because it is needed for punctured surfaces.

\begin{remark}
In the complexity-one cases, one can also verify
\Cref{thm:level_set} directly using explicit formulas
for extremal length (e.g. \cite{Miyachi2017},
\cite{FortierBourqueEtAl2024}).
\end{remark}

\subsection{Analyticity and rigidity}
\label{subsec:lemmas}

The \emph{rational depth} of $\lambda\in\ml$ is the dimension of the space of rational linear functions that vanish at $\lambda$ in a train-track coordinate chart containing $\lambda$ (cf.\ \cite[Definition 9.5.10]{Thurston-LectureNote}). A measured lamination of rational depth zero is essentially complete (cf.\ \cite[Proposition 9.5.11]{Thurston-LectureNote}). Because the Thurston measure belongs to the Lebesgue measure class in each train-track chart and the train-track atlas is countable, the set of measured laminations of rational depth zero has full $\ThursM$-measure in $\ml$.

\begin{lemma}[Local real analyticity of the
Hubbard--Masur map]
\label{lem:HM_local_analytic}
Let $x\in\teich_{g,m}$, and let $q_0\in\mathcal Q_x\setminus \{0\}$
have only simple zeros and a simple pole at every
puncture.
Then there exist a neighborhood $U$ of $q_0$ and a
train-track coordinate chart such that the
Hubbard--Masur map restricted to $U$ is a
real-analytic diffeomorphism onto an open subset
of that chart.
\end{lemma}

\begin{proof}
For closed surfaces, this follows from Dumas's
adapted train-track construction
\cite[Lemma 5.7 and Theorem 5.8]{Dumas2015}.

For punctured surfaces, the same construction applies,
with simple poles treated as singularities of the
flat metric.
After filling in the punctures, an integrable
quadratic differential extends meromorphically
with at most simple poles.
A simple pole gives a cone singularity of angle $\pi$.
The Delaunay triangulation construction of
Masur and Smillie applies to these finite-area
flat metrics
\cite[\S4]{MasurSmillie1991}.
The complementary region of the adapted train track
corresponding to a simple pole is a punctured monogon.

On the canonical double cover, the square root of
the quadratic differential extends holomorphically
across the preimages of the simple poles.
Thus the relative-period and symplectic arguments
in Dumas's proof apply without additional pole terms,
giving the required local real-analytic
diffeomorphism.
\end{proof}

\begin{lemma}[Common train track]
\label{lem:common_splitting}
Let $x_1,x_2\in\teich_{g,m}$, and let $\lambda_0\in\ml$ be a measured lamination of rational depth zero. Then there exist a complete train track $\tau$
(in the sense of \cite{PennerHarer1992}), a strictly positive transverse measure $\mu_0\in E(\tau)$ representing $\lambda_0$, and a neighborhood $U$ of $\mu_0$ in $W(\tau)$ such that $U\subset E(\tau)$ and, for each $i=1,2$, the map
\[
U\ni\mu\longmapsto q_{R_\tau(\mu),x_i}\in\mathcal Q_{x_i}
\]
is a real-analytic diffeomorphism onto an open neighborhood of $q_{\lambda_0,x_i}$. Here $R_\tau$ is the realizing map in \eqref{eq:realizing_map}.
\end{lemma}

\begin{proof}
Set $q_i=q_{\lambda_0,x_i}$ for $i=1,2$.
Since $\lambda_0$ has rational depth zero, it is
essentially complete, and each $q_i$ belongs to
the principal stratum.
By \Cref{lem:HM_local_analytic}, choose positively
measured generic train tracks $(\tau_i,\mu_i)$
representing $\lambda_0$ and neighborhoods $U_i$
of $\mu_i$ such that
\[
U_i\ni\mu\longmapsto q_{R_{\tau_i}(\mu),x_i}
\]
is a real-analytic diffeomorphism onto an open neighborhood of $q_i$.

By \cite[Theorem 2.8.5]{PennerHarer1992}, the measured train tracks $(\tau_1,\mu_1)$ and $(\tau_2,\mu_2)$ are equivalent. The common-subtrack theorem then provides a generic train track $\tau_3$ equipped with a positive transverse measure $\mu_3$. After any necessary shifts, each $(\tau_i,\mu_i)$ is related to $(\tau_3,\mu_3)$ by a finite sequence of measured splits (cf.\ \cite[Theorem 2.3.1]{PennerHarer1992}).

Neither splitting sequence can have a collision. At the first collision, the equality of the two relevant branch weights would yield a nontrivial rational linear relation among the weights representing $\lambda_0$, independent of the switch equations. Pulling this relation back through the preceding shifts and non-collision splits would imply that $\lambda_0$ has positive rational depth, contrary to our assumption. Thus, $W(\tau_3)$ and $W(\tau_i)$ have the same dimension.

Let $C_i\colon W(\tau_3)\to W(\tau_i)$ be the linear carrying map induced by the splitting sequence. Since carrying preserves the represented measured lamination, $R_{\tau_i}\circ C_i=R_{\tau_3}$. The injectivity of $R_{\tau_3}$ then implies that $C_i$ is injective. Because the source and target have the same dimension, $C_i$ is an isomorphism.

Choose a small connected neighborhood $U$ of $\mu_3$ in the positive weight cone, with $U\subset C_1^{-1}(U_1)\cap C_2^{-1}(U_2)$. For each $i=1,2$, the map
\[
U\ni\mu\longmapsto q_{R_{\tau_3}(\mu),x_i}
=q_{R_{\tau_i}(C_i(\mu)),x_i}
\]
is a real-analytic diffeomorphism onto an open neighborhood of $q_i$. In particular, $R_{\tau_3}(U)$ is an open neighborhood of $\lambda_0$ in $\ml$. By the standard characterization of complete train tracks, $\tau_3$ is therefore complete (cf.\ \cite[Propositions 1.4.9 and 1.7.6]{PennerHarer1992}). Taking $\tau=\tau_3$ and $\mu_0=\mu_3$ proves the lemma.
\end{proof}

\begin{lemma}
\label{lem:coincidence}
Let $x_1,x_2\in\teich_{g,m}$ and $c\in\mathbb R$. If $\ext_{x_1}(\lambda)=e^c\ext_{x_2}(\lambda)$ for every $\lambda$ in a nonempty open set $V\subset\ml$, then $c=0$ and $x_1=x_2$.
\end{lemma}

\begin{proof}
Choose a measured lamination $\lambda_0\in V$ of rational depth zero. By \Cref{lem:common_splitting}, there exist a complete train track $\tau$, a strictly positive measure $\mu_0\in E(\tau)$ representing $\lambda_0$, and a neighborhood $U$ of $\mu_0$ such that $R_\tau(U)\subset V$ and the conclusions of that lemma hold.

Suppose first that $c=0$.
Since $\lambda_0$ is essentially complete, its support is maximal. For each $i=1,2$, let $\Sigma_{\lambda_0,\tau}(x_i)\in W(\tau)$ be the element defined in \cite[\S6.2]{MiyachiOhshika2017}. Let $\omega_{Th}$ denote the Thurston symplectic form on $W(\tau)$.
The differential formula for extremal length gives
\[
d_{\mu_0}(\ext_{x_i}\circ R_\tau)(\xi)
=2\omega_{Th}\bigl(\xi,\Sigma_{\lambda_0,\tau}(x_i)\bigr)
\]
for every $\xi\in W(\tau)$ (cf.\ \cite[Theorem 1.1]{Miyachi2013b} and \cite[(1.1) and (6.5)]{MiyachiOhshika2017}). The two extremal length functions agree on $R_\tau(U)$, so their differentials at $\mu_0$ agree. By the nondegeneracy of the Thurston symplectic form,
\[
\Sigma_{\lambda_0,\tau}(x_1)
=\Sigma_{\lambda_0,\tau}(x_2).
\]
The map $x\mapsto\Sigma_{\lambda_0,\tau}(x)$ is an embedding by \cite[Th\'eor\`eme 6.1]{MiyachiOhshika2017}, so $x_1=x_2$.

Now suppose that $c\ne0$. The same argument gives
\[
\Sigma_{\lambda_0,\tau}(x_1)
=e^c\Sigma_{\lambda_0,\tau}(x_2).
\]
Let $F_i$ be the horizontal foliation of $q_{\lambda_0,x_i}$.
From the proof of \cite[Th\'eor\`eme 6.1]{MiyachiOhshika2017}, we have
\begin{align*}
i(\gamma,F_1)=\omega_{Th}(\gamma,\Sigma_{\lambda_0,\tau}(x_1))=e^c\omega_{Th}(\gamma,\Sigma_{\lambda_0,\tau}(x_2))
=e^ci(\gamma,F_2)=i(\gamma,e^cF_2)
\end{align*}
for sufficiently many simple closed curves $\gamma\in E(\tau)$, so $F_1=e^cF_2$.
Thus, $x_1$ and $x_2$ lie on a Teichm\"uller geodesic
with direction $[\lambda_0]$, and
$d_T(x_1,x_2)=|c|/2>0$.

The same argument applies to every measured lamination
of rational depth zero in $V$.
The Teichm\"uller geodesic joining the distinct points
$x_1$ and $x_2$ is unique. Hence the projective classes
of all such laminations belong to the two-element set
determined by the vertical and horizontal foliations
of this geodesic. Therefore, the measured laminations
of rational depth zero in $V$ lie in at most two rays
in $\ml$.

Each of these rays has zero Thurston measure.
However, measured laminations of rational depth zero
have full Thurston measure in $V$, and the nonempty
open set $V$ has positive Thurston measure. This is
a contradiction.
\end{proof}

\subsection{Proof of \Cref{thm:level_set}}
Fix distinct points $x_1,x_2\in\teich_{g,m}$ and
$c\in\mathbb R$. Let $\mathcal R\subset\ml$ denote
the full-measure set of measured laminations of
rational depth zero.

For each $\lambda_0\in\mathcal R$,
\Cref{lem:common_splitting} provides a complete train
track $\tau$ and a connected open neighborhood $U$
of the weight representing $\lambda_0$ on which both maps
\[
\mu\longmapsto q_{R_\tau(\mu),x_i},
\qquad i=1,2,
\]
are real analytic.
After shrinking $U$, we may assume that both
$q_{R_\tau(\mu),x_1}$ and $q_{R_\tau(\mu),x_2}$
belong to the principal stratum for every $\mu\in U$.
The function
\[
F(\mu)
=
\ext_{x_1}(R_\tau(\mu))
-e^c\ext_{x_2}(R_\tau(\mu))
\]
is therefore real analytic on $U$, since the
$L^1$-norm is real analytic on the principal stratum
(cf.\ \cite[\S\S5.3--5.4]{Dumas2015}).

By \Cref{lem:coincidence}, $F$ cannot vanish
identically on $U$.
Its zero set consequently has Lebesgue measure
zero \cite{Mityagin2020}.
Since the Thurston measure belongs to the Lebesgue
measure class in train-track coordinates,
\[
\ThursM\bigl(
\ml_c(x_1,x_2)\cap R_\tau(U)
\bigr)=0.
\]

The sets $R_\tau(U)$ form an open cover of $\mathcal R$.
Since $\ml$ is second countable, so is its subspace
$\mathcal R$. Thus, this cover has a countable subcover.
Each corresponding intersection
$\ml_c(x_1,x_2)\cap R_\tau(U)$ is null, so
$\ml_c(x_1,x_2)\cap\mathcal R$ is null.
Since $\ml\setminus\mathcal R$ is also null,
$\ml_c(x_1,x_2)$ has zero Thurston measure.

Finally, since extremal length is homogeneous of degree
two, the cone over $\pml_c(x_1,x_2)$ within
$\mathcal{BML}_x$ is contained in $\ml_c(x_1,x_2)$.
By the definition of the projectivized Thurston
measure in \eqref{eq:def_Th_measure}, we obtain
\[
\PThursM^x(\pml_c(x_1,x_2))=0.
\]
This proves \Cref{thm:level_set}.

\section{The Hopf decomposition for subgroup actions}
\label{sec:Hopf_decomposition}
In this section, we prove \Cref{thm:Hopf_decomposition_main,thm:main_classfication_limit_sets}.
After recalling the measure-theoretic terminology, we establish the
Radon--Nikodym description of the Hopf decomposition and identify its
conservative part with the big horospherical limit set. We then use
Dirichlet polyhedra to identify the dissipative part. All null sets on
$\pml$ are understood with respect to the Thurston measure class.

\subsection{Conservativity and dissipativity}
\label{subsec:conservativity_dissipativity}
Let $(X,P)$ be a standard probability space, completed if necessary.
Suppose that a countable group $\Gamma$ acts nonsingularly on $X$;
that is, $P$ is quasi-invariant under the action.

A measurable subset $B\subset X$ is called a \emph{wandering set} if $P(B\cap\gamma(B))=0$ for every $\gamma\in\Gamma\setminus\{\mathrm{id}\}$.
The action on a $\Gamma$-invariant measurable subset $Z\subset X$ is called \emph{dissipative} if there exists a measurable wandering set $B\subset Z$ such that
\[
P\left(Z\setminus\bigcup_{\gamma\in\Gamma}\gamma(B)\right)=0.
\]
The action on a $\Gamma$-invariant measurable subset $Y\subset X$ is called \emph{conservative} if, for every measurable subset $B\subset Y$ of positive measure, $P(B\cap\gamma(B))>0$ for infinitely many $\gamma\in\Gamma$.

We call the action of a subgroup $G\le\tmod(g,m)$
\emph{conservative} or \emph{totally dissipative} if it is
conservative or dissipative, respectively, on all of $\pml$.
For background, see \cite{Pommerenke1976}, \cite{Sullivan1981},
and \cite{Nicholls1989}, as well as \cite{Kaimanovich2010} and
\cite[Definition D11.3]{HubbardVol3_2022}.

\begin{remark}
Our terminology differs slightly from that of \cite{Kaimanovich2010}.
Our conservativity corresponds to his notion of \emph{infinite
conservativity}, and our dissipativity corresponds to his notion of
\emph{complete dissipativity}. In his terminology, an action is
conservative if every measurable set of positive measure is recurrent.
\end{remark}

An action of $\Gamma$ on a measure space $(X,\mu)$ is
\emph{essentially free} if the fixed-point set
$\operatorname{Fix}(\gamma)=\{x\in X\mid \gamma(x)=x\}$ is null
for every $\gamma\in \Gamma\setminus \{\mathrm{id}\}$.
For essentially free actions, Kaimanovich's conservative and
infinitely conservative parts agree modulo null sets. By
\Cref{lem:Fix_null}, the subgroup actions on $\pml$ considered below
are essentially free, so this distinction does not arise.

\subsection{Proof strategy}
\label{subsec:Main_resut_of_SectionDichotomy}
Recall that the introduction defines $\conservativepart(G)$ and
$\dissipativepart(G)$ by divergence and convergence, respectively, of
the sum of the Radon--Nikodym derivatives. Their independence of $x$ and their
$G$-invariance are established in
\Cref{lem:Cons-Diss-ElementaryProperty}.

We prove \Cref{thm:Hopf_decomposition_main} in
\S\ref{subsec:ProofOfDecomposition}. The identification of the
conservative part with $\Lambda_H(G)$, the first assertion of
\Cref{thm:main_classfication_limit_sets}, is proved in
\S\ref{subsec:Proof_of_ConservativePart_BigHorospherical}.
For the dissipative part, we prove the following restatement of the
second assertion of \Cref{thm:main_classfication_limit_sets}. It identifies
the dissipative part with both the Dirichlet set and the union of the
translates of the ideal boundary of a Dirichlet polyhedron, paralleling
Sullivan's description for Kleinian groups
(cf.\ \cite[Corollary in Section IV]{Sullivan1981}).

\begin{theorem}[The dissipative part and Dirichlet polyhedra]
\label{thm:Dissipative_part}
Let $G$ be a nontrivial subgroup of the Teichm\"uller modular group,
and let $x_0\in\teich_{g,m}$ satisfy
$\operatorname{Stab}_G(x_0)=\{\mathrm{id}\}$. Then the symmetric
differences
\[
\dissipativepart(G)\bigtriangleup \bigcup_{[\omega]\in G}[\omega](\Dirichlet^\infty_G(x_0)),\quad
\dissipativepart(G)\bigtriangleup \operatorname{Dir}_G(x_0)
\]
are null with respect to $\PThursM^{x}$ for all $x\in \teich_{g,m}$.
\end{theorem}

The set $\Dirichlet^\infty_G(x_0)$ is the ideal boundary of the
Dirichlet polyhedron $\Dirichlet_G(x_0)$ centered at $x_0$, discussed
in \S\ref{subsec:DirichletDomain_Ideal}. We show that this ideal boundary
is wandering and that almost every point of it is a Dirichlet point
(see \Cref{lem:almost-unique}). The proof of
\Cref{thm:Dissipative_part} is given in
\S\ref{subsec:Proof-Dissipative_part}.

\begin{remark}
The proof of \cite[Proposition D11.5]{HubbardVol3_2022} shows that,
if $\operatorname{Stab}_G(x)=\{\mathrm{id}\}$, the action of $G$ on
$\dissipativepart(G)$ is dissipative with respect to $\PThursM^x$.
Although that proposition is stated for torsion-free groups, its proof
of dissipativity requires only a trivial stabilizer at the basepoint
in the present setting. Moreover,
\Cref{thm:Hopf_decomposition_main} shows that
$\dissipativepart(G)$ is maximal among dissipative invariant subsets,
modulo null sets.
\end{remark}

\subsection{Preliminary lemmas for the proof of \Cref{thm:Hopf_decomposition_main}}
The first lemma establishes basepoint independence and $G$-invariance
of the two parts. The second proves essential freeness, which allows us
to apply Kaimanovich's criterion.
\begin{lemma}
\label{lem:Cons-Diss-ElementaryProperty}
Let $G$ be a subgroup of the Teichm\"uller modular group.
\begin{itemize}
\item[(a)]
The sets $\conservativepart(G)$ and $\dissipativepart(G)$ do not depend on the choice of basepoint $x\in \teich_{g,m}$.
\item[(b)]
$\conservativepart(G)$ and $\dissipativepart(G)$ are invariant under
the action of $G$.
\end{itemize}
\end{lemma}

\begin{proof}
\textbf{(a).}
By \eqref{eq:QC-inv_extremal_length} and the invariance of $d_T$
under the action of $G$, for all $x,y\in\teich_{g,m}$ and
$[\omega]\in G$ we have
\[
e^{-4\convgenus d_T(x,y)}
\left(
\frac{\ext_x(\lambda)}
{\ext_{[\omega](x)}(\lambda)}
\right)^{\convgenus}
\le
\left(
\frac{\ext_y(\lambda)}
{\ext_{[\omega](y)}(\lambda)}
\right)^{\convgenus}
\le
e^{4\convgenus d_T(x,y)}
\left(
\frac{\ext_x(\lambda)}
{\ext_{[\omega](x)}(\lambda)}
\right)^{\convgenus}.
\]
Summing over $[\omega]\in G$ shows that the defining series converges
or diverges independently of the basepoint.

\medskip
\noindent
\textbf{(b).}
For $[\eta]\in G$, equivariance of extremal length gives
\begin{align*}
\sum_{[\omega]\in G}
\left(
\frac{\ext_x([\eta](\lambda))}
     {\ext_{[\omega](x)}([\eta](\lambda))}
\right)^{\convgenus}
& =\sum_{[\omega]\in G}
\left(
\frac{\ext_{[\eta]^{-1}(x)}(\lambda)}
     {\ext_{[\eta]^{-1}\circ[\omega](x)}(\lambda)}
\right)^{\convgenus} \\
& =\sum_{[\gamma]\in G}
\left(
\frac{\ext_{[\eta]^{-1}(x)}(\lambda)}
     {\ext_{[\gamma]\circ[\eta]^{-1}(x)}(\lambda)}
\right)^{\convgenus},
\end{align*}
where $[\gamma]=[\eta]^{-1}\circ[\omega]\circ[\eta]$.
Part~(a) now implies that convergence or divergence is unchanged by
the action of $[\eta]$, proving~(b).
\end{proof}

\begin{lemma}[Essential freeness]
\label{lem:Fix_null}
The action of $\tmod(g,m)$ on $\pml$ is essentially free: for every
$[\omega]\in \tmod(g,m)\setminus \{\mathrm{id}\}$,
\[
\operatorname{Fix}([\omega])
=
\{[\lambda]\in\pml\mid
[\omega]([\lambda])=[\lambda]\}
\]
is $\PThursM^x$-null for every $x\in\teich_{g,m}$.
\end{lemma}

\begin{proof}
Suppose first that $[\omega]$ has finite order. Since the action of
$\tmod(g,m)$ on $\teich_{g,m}$ is faithful and
$[\omega]\ne\mathrm{id}$, there is $x_0\in\teich_{g,m}$ such that
$[\omega](x_0)\ne x_0$. Choose $k\ge1$ such that
$[\omega]^k=\mathrm{id}$. For $[\lambda]\in\operatorname{Fix}([\omega])$,
\[
[\omega](\lambda)=c\lambda
\]
for some $c>0$. Since $[\omega]^k(\lambda)=c^k\lambda=\lambda$,
we have $c=1$.

By equivariance of extremal length,
\begin{align*}
\operatorname{Fix}([\omega])
&\subset
\left\{
[\lambda]\in\pml\mathrel{}\middle|\mathrel{}
\ext_{x_0}(\lambda)
=
\ext_{x_0}([\omega](\lambda))
\right\} \\
&=
\left\{
[\lambda]\in\pml\mathrel{}\middle|\mathrel{}
\ext_{x_0}(\lambda)
=
\ext_{[\omega]^{-1}(x_0)}(\lambda)
\right\} \\
&=
\pml\bigl(x_0,[\omega]^{-1}(x_0)\bigr).
\end{align*}
By \Cref{thm:level_set}, this set is $\PThursM^{x_0}$-null. Mutual
absolute continuity of the measures $\PThursM^x$ then shows that it is
$\PThursM^x$-null for every $x\in\teich_{g,m}$.

Now suppose that $[\omega]$ has infinite order. If $[\omega]$ is
pseudo-Anosov, then it has exactly two fixed points in $\pml$, namely
the projective classes of its stable and unstable measured laminations.
Since $\PThursM^x$ has no atoms, its fixed-point set is null.

Suppose that $[\omega]$ is reducible. Choose $n\ge1$ such that
$[\varphi]=[\omega]^n$ fixes an essential reduction class $\alpha$ (cf. \cite[\S2]{BirmanLubotzkyMcCarthy1983}).
Because $[\omega]$ has infinite order, $[\varphi]$ still acts
nontrivially on $\teich_{g,m}$. Moreover,
\[
\operatorname{Fix}([\omega])
\subset \operatorname{Fix}([\varphi]).
\]

Let $[\lambda]\in\operatorname{Fix}([\varphi])$, and write
\[
[\varphi](\lambda)=c\lambda
\]
for some $c>0$. Since $[\varphi](\alpha)=\alpha$, invariance of the
intersection number gives
\[
i(\alpha,\lambda)
=
i([\varphi](\alpha),[\varphi](\lambda))
=
c\,i(\alpha,\lambda).
\]
Thus either $i(\alpha,\lambda)=0$ or $c=1$, and hence
\[
\operatorname{Fix}([\omega])
\subset
\{[\lambda]\in\pml\mid i(\alpha,\lambda)=0\}
\cup
\pml\bigl(x_0,[\varphi]^{-1}(x_0)\bigr),
\]
where $x_0$ is chosen so that $[\varphi](x_0)\ne x_0$
(cf.\ \cite[Corollary~2.17]{Ivanov1992}).

The first set is null because filling measured laminations have full
Thurston measure, and the second is null by \Cref{thm:level_set}.
Therefore, $\operatorname{Fix}([\omega])$ is
$\PThursM^x$-null for every $x\in\teich_{g,m}$.
\end{proof}

\subsection{Proof of \Cref{thm:Hopf_decomposition_main}}
\label{subsec:ProofOfDecomposition}
We recall the relevant terminology from \cite{Kaimanovich2010}.
Let $\mathscr{C}$ and $\mathscr{D}$ denote the unions of the
ergodic components whose conditional probability measures are
purely nonatomic and purely atomic, respectively. Following Kaimanovich
\cite{Kaimanovich2010}, we call these the \emph{continual} and
\emph{discontinual} parts of the action of $G$.
Set
$\mathscr{D}_{\infty}=\{[\lambda]\in\mathscr{D}:\#\operatorname{Stab}_G([\lambda])=\infty\}$,
\[
\mathtt{cons}_{\infty}
=
\mathscr{C}\cup\mathscr{D}_{\infty}.
\]
Also define
$A_0=
\{[\lambda]\in\pml:
\operatorname{Stab}_G([\lambda])\ne\{\mathrm{id}\}\}$.
By \cite[Theorem~14 and Corollary~15]{Kaimanovich2010},
$\mathtt{cons}_{\infty}$ is the infinitely conservative part,
whereas $\mathtt{cons}_{\infty}\cup A_0$ is the conservative part
in Kaimanovich's terminology, modulo null sets.

Fix $x\in\teich_{g,m}$ and first suppose that $G$ is infinite.
By \eqref{eq:Radon-Nikodym2} and
\cite[Theorem~29]{Kaimanovich2010},
$\conservativepart(G)$ agrees with
$\mathtt{cons}_{\infty}$ modulo null sets.
The set $A_0$ is null by \Cref{lem:Fix_null} and the countability of
$G$. Thus the action on $\conservativepart(G)$ is conservative in our
terminology, while the action on its complement
$\dissipativepart(G)$ is dissipative. The result follows from the
Hopf decomposition for general nonsingular actions
\cite[Theorem~14 and Corollary~15]{Kaimanovich2010}.

If $G$ is finite, the defining series converges everywhere, so
$\conservativepart(G)=\varnothing$ and $\dissipativepart(G)=\pml$.
By essential freeness, we may restrict the action to a $G$-invariant
conull Borel subset on which it is free. A free action of a finite
group on a standard Borel space admits a Borel fundamental domain.
Thus the action is totally dissipative.
\qed

% \begin{remark}
% The proof of \Cref{thm:Hopf_decomposition_main} also describes the
% ergodic components of the action of a nontrivial subgroup
% $G$ of $\tmod(g,m)$:
% \begin{align*}
% (\text{the conservative part of the action of $G$})&=\mathscr{C} \\
% (\text{the dissipative part of the action of $G$})&= \mathscr{D}_1=\mathscr{D}\setminus A_0
% \end{align*}
% modulo $\PThursM^x$-null sets. Essential freeness
% (\Cref{lem:Fix_null}) implies that both $\mathscr{D}_\infty$ and $A_0$
% are null. See also \cite[Corollary 16]{Kaimanovich2010}.
% \end{remark}

\subsection{The conservative part}
\label{subsec:Proof_of_ConservativePart_BigHorospherical}
We now prove the first assertion of
\Cref{thm:main_classfication_limit_sets}.

If $G$ is finite, both $\conservativepart(G)$ and $\Lambda_H(G)$ are
empty. Hence, we may assume that $G$ is infinite.

Fix $x\in\teich_{g,m}$. 
By \cite[Theorem~29]{Kaimanovich2010}, the infinitely conservative
part $\mathtt{cons}_\infty$ agrees modulo null sets with
\[
\left\{
[\lambda]\in\pml
\mathrel{}\middle|\mathrel{}
\exists a>0\text{ s.t. }
\#\left\{[\omega]\in G:\frac{d([\omega]_*\PThursM^x)}
     {d\PThursM^x}([\lambda])>a\right\}=\infty
\right\}.
\]
For $[\lambda]\in\pml^{ue}$, equations
\eqref{eq:Radon-Nikodym2} and \eqref{eq:busemann_ue} show that this
condition is equivalent to the existence of $b\in\mathbb{R}$ such that
\[
\#\left\{
[\omega]\in G
\mathrel{}\middle|\mathrel{}
\busemann^x_\lambda([\omega](x))<b
\right\}
=\infty.
\]
Since $\operatorname{Stab}_G(x)$ is finite, this condition says exactly
that some horoball centered at $[\lambda]$ contains infinitely many
points of the orbit $G(x)$. In other words, it is equivalent to
$[\lambda]\in\Lambda_H(G)$.

The set $\pml^{ue}$ has full $\PThursM^x$-measure. Moreover, by
\Cref{lem:Fix_null}, the action of $G$ on $\pml$ is essentially free.
Consequently, its conservative and infinitely conservative parts agree
modulo null sets. It follows that
\[
\conservativepart(G)=\Lambda_H(G)
\]
modulo $\PThursM^x$-null sets.
\qed

\subsection{Ideal boundary of the Dirichlet polyhedron}
\label{subsec:DirichletDomain_Ideal}
McCarthy and Papadopoulos construct fundamental domains for subgroups
of the mapping class group that are analogous to Dirichlet polyhedra
for hyperbolic manifolds \cite{McCarthyPapadopoulos1996}; see, for
example, \cite{Beardon1995}.

Let $G\le \tmod(g,m)$ and let $x_0\in\teich_{g,m}$ satisfy
$\operatorname{Stab}_G(x_0)=\{\mathrm{id}\}$. We define the
\emph{Dirichlet polyhedron} $\Dirichlet_G(x_0)$ for $G$, centered at
$x_0$, by
\[
\Dirichlet_G(x_0)
=\left\{x\in\teich_{g,m}\mathrel{}\middle|\mathrel{}
d_T(x,x_0)<d_T(x,[\omega](x_0))
\text{ for every }[\omega]\in G\setminus\{\mathrm{id}\}\right\}.
\]
McCarthy and Papadopoulos establish the following properties:
\begin{itemize}
\item The closure of the Dirichlet polyhedron in $\teich_{g,m}$ is
given by
\[
\overline{\Dirichlet_G(x_0)}
=\left\{x\in\teich_{g,m}\mathrel{}\middle|\mathrel{}
d_T(x,x_0)\le d_T(x,[\omega](x_0))
\text{ for every }[\omega]\in G\right\},
\]
and $x_0$ is an interior point of this closure (\cite[Propositions 5.1 and 5.2]{McCarthyPapadopoulos1996}).
\item The Dirichlet polyhedron $\Dirichlet_G(x_0)$ and its closure are
starlike with respect to $x_0$
(\cite[Proposition 5.3]{McCarthyPapadopoulos1996}).
\item As in hyperbolic spaces (see \cite[Chapter 9]{Beardon1995}),
$\Dirichlet_G(x_0)$ is a locally finite fundamental domain in
$\teich_{g,m}$ (\cite[Theorem 5.9]{McCarthyPapadopoulos1996}).
\item The boundary $\partial \Dirichlet_G(x_0)$ is a properly
embedded, locally flat submanifold of codimension one in
$\teich_{g,m}$. Every boundary point lies in at least one hypersurface
of the form
\[
L_{[\omega]}(x_0)
=\{x\in \teich_{g,m}\mid d_T(x,x_0)=d_T(x,[\omega](x_0))\}
\]
for some $[\omega]\in G\setminus\{\mathrm{id}\}$ (\cite[Propositions 5.2 and 5.5]{McCarthyPapadopoulos1996}).
\end{itemize}
Let $\Dirichlet^\infty_G(x_0)$ denote the intersection of the closure
of $\Dirichlet_G(x_0)$ in the Thurston compactification with $\pml$.

\begin{lemma}
\label{lem:Dirichlet_geodesic_horoball}
Let $[\lambda]\in\pml^{ue}$. Consider the following conditions:
\begin{itemize}
\item[(a)] The Teichm\"uller ray $\ray^{\lambda}_{x_0}$ is contained in $\Dirichlet_G(x_0)$.
\item[(b)] $[\lambda]\in \Dirichlet^\infty_G(x_0)$.
\item[(c)] The $0$-horoball $\horos([\lambda];x_0)$ is disjoint from the $G$-orbit $G(x_0)$ of $x_0$.
\end{itemize}
Condition ${\rm (a)}$ implies ${\rm (c)}$, while conditions
${\rm (a)}$ and ${\rm (b)}$ are equivalent. Moreover, if
$\overline{\horos([\lambda];x_0)}\cap G(x_0)=\{x_0\}$,
then condition ${\rm (a)}$ holds.
\end{lemma}

\begin{proof}
\noindent
\textbf{${\rm (a)}\Rightarrow {\rm (c)}$.}
Let $[\omega]\in G$. Because
$\ray^\lambda_{x_0}(t)\in \Dirichlet_G(x_0)$ for every $t\ge 0$,
\[
d_T(\ray^\lambda_{x_0}(t),[\omega](x_0))
\ge d_T(\ray^\lambda_{x_0}(t),x_0)=t,
\]
and hence
\[
\busemann^{x_0}_\lambda([\omega](x_0))
=\lim_{t\to\infty}
\bigl(d_T(\ray^\lambda_{x_0}(t),[\omega](x_0))-t\bigr)
\ge 0.
\]
Thus $[\omega](x_0)$ lies outside the $0$-horoball.

We now prove the additional assertion. Suppose that $x_0$ is the only
orbit point in the closure of the $0$-horoball, and let
$[\omega]\in G\setminus\{\mathrm{id}\}$. Then
$\busemann^{x_0}_\lambda([\omega](x_0))>0$. The function
$t\mapsto d_T([\omega](x_0),\ray^\lambda_{x_0}(t))-t$ is
nonincreasing and converges to
$\busemann^{x_0}_\lambda([\omega](x_0))$. Therefore, for every
$t\ge0$,
\[
d_T([\omega](x_0),\ray^\lambda_{x_0}(t))-t
\ge\busemann^{x_0}_\lambda([\omega](x_0))>0.
\]
Since this holds for every $[\omega]\in G\setminus\{\mathrm{id}\}$,
$\ray^\lambda_{x_0}(t)\in\Dirichlet_G(x_0)$ for every $t\ge0$.
Hence ${\rm (a)}$ holds.

\medskip
\noindent
\textbf{${\rm (b)}\Rightarrow {\rm (a)}$.}
Since $[\lambda]\in \Dirichlet^\infty_G(x_0)$, there is a sequence
$\{x_n\}_{n=1}^\infty\subset \Dirichlet_G(x_0)$ converging to
$[\lambda]$ in the Thurston compactification. Set
$t_n=d_T(x_0,x_n)$. Then $x_n$ also converges to $[\lambda]$ in the
Gardiner--Masur compactification, and
\[
e^{-2t_n}\ext_{x_n}(\cdot)
\longrightarrow \frac{i(\cdot,\lambda)^2}{\ext_{x_0}(\lambda)}
\]
uniformly on compact subsets of $\ml$ (cf.\ \cite[Corollaries 1 and 2]{Miyachi2013}).

For each $n\in\mathbb{N}$, choose $[\lambda_n]\in\pml$ such that
$\ray^{\lambda_n}_{x_0}(t_n)=x_n$, and normalize its representative
by $\ext_{x_0}(\lambda_n)=1$. Then $t_n\to\infty$. After passing
to a subsequence, we may assume that $\lambda_n$ converges to some
$\lambda_0\in\ml$. Moreover,
\[
e^{-2t_n}\ext_{x_n}(\lambda_n)
=e^{-4t_n}\ext_{x_0}(\lambda_n)\longrightarrow 0.
\]
Hence $i(\lambda_0,\lambda)=0$. Since $\lambda$ is uniquely ergodic,
we have $[\lambda_0]=[\lambda]$. By Teichm\"uller's theorem, the rays
$\ray^{\lambda_n}_{x_0}$ converge to $\ray^\lambda_{x_0}$ uniformly on
compact subsets of $[0,\infty)$.

Suppose, for contradiction, that the ray is not contained in
$\Dirichlet_G(x_0)$. Since $x_0\in\Dirichlet_G(x_0)$ and the ray is
continuous, there is $t_0>0$ such that
$\ray^\lambda_{x_0}(t_0)\in\partial\Dirichlet_G(x_0)$. Thus there
exists $[\omega]\in G\setminus\{\mathrm{id}\}$ such that
\[
d_T(\ray^\lambda_{x_0}(t_0),x_0)
=d_T(\ray^\lambda_{x_0}(t_0),[\omega](x_0)).
\]
By \cite[Lemma 4.5]{McCarthyPapadopoulos1996},
\[
d_T(x_0,\ray^\lambda_{x_0}(t_0+1))
>d_T([\omega](x_0),\ray^\lambda_{x_0}(t_0+1)).
\]
For all sufficiently large $n$, we therefore have $t_0+1<t_n$ and
$\ray^{\lambda_n}_{x_0}(t_0+1)\notin\Dirichlet_G(x_0)$. This
contradicts the starlikeness of $\Dirichlet_G(x_0)$ with respect to
$x_0$, since $x_n\in\Dirichlet_G(x_0)$.

\medskip
\noindent
\textbf{${\rm (a)}\Rightarrow {\rm (b)}$.}
Since $\lambda$ is uniquely ergodic,
$\ray^\lambda_{x_0}(t)\to[\lambda]$ in the Thurston compactification
(cf.\ \cite{Masur1980}). By ${\rm (a)}$, the entire ray lies in
$\Dirichlet_G(x_0)$, so $[\lambda]\in\Dirichlet^\infty_G(x_0)$.
\end{proof}

\begin{lemma}[Wandering property and uniqueness]
\label{lem:almost-unique}
Assume that $[\omega](x_0)\ne x_0$ for every
$[\omega]\in G\setminus\{\mathrm{id}\}$. Then
$\Dirichlet^\infty_G(x_0)$ is wandering under the action of $G$.
Moreover, for $\PThursM^{x_0}$-almost every
$[\lambda]\in\Dirichlet^\infty_G(x_0)$,
$\overline{\horos([\lambda];x_0)}\cap G(x_0)=\{x_0\}$.
\end{lemma}

\begin{proof}
Fix $[\omega]\in G\setminus\{\mathrm{id}\}$, and let
$[\lambda]\in \Dirichlet^\infty_G(x_0)\cap
[\omega](\Dirichlet^\infty_G(x_0))\cap \pml^{ue}$.
Since $[\omega](\Dirichlet^\infty_G(x_0))=
\Dirichlet^\infty_G([\omega](x_0))$, \Cref{lem:Dirichlet_geodesic_horoball} gives
\[
\busemann^{x_0}_\lambda([\omega](x_0))\ge0,
\qquad
\busemann^{[\omega](x_0)}_\lambda(x_0)\ge0.
\]
By the cocycle condition \eqref{eq:busemann_cocycle_condition},
\[
0=\busemann^{x_0}_\lambda(x_0)
=\busemann^{x_0}_\lambda([\omega](x_0))
+\busemann^{[\omega](x_0)}_\lambda(x_0).
\]
Both summands are nonnegative, so
$\busemann^{x_0}_\lambda(x_0)
=\busemann^{x_0}_\lambda([\omega](x_0))=0$.
Therefore,
\begin{align*}
\frac{\ext_{[\omega](x_0)}(\lambda)}{\ext_{x_0}(\lambda)}
&=\exp\left(2\busemann^{x_0}_\lambda([\omega](x_0))\right)=\exp\left(2\busemann^{x_0}_\lambda(x_0)\right)=1.
\end{align*}
Thus,
\[
\Dirichlet^\infty_G(x_0)\cap
[\omega](\Dirichlet^\infty_G(x_0))\cap\pml^{ue}
\subset\pml(x_0,[\omega](x_0)).
\]
The right-hand side is null by \Cref{thm:level_set}, and
$\pml\setminus\pml^{ue}$ is also null. Thus
$\Dirichlet^\infty_G(x_0)$ is wandering.

Let $[\lambda]\in\Dirichlet^\infty_G(x_0)\cap\pml^{ue}$.
Suppose that $[\omega](x_0)\in\overline{\horos([\lambda];x_0)}$
for some $[\omega]\in G\setminus\{\mathrm{id}\}$.
By \Cref{lem:Dirichlet_geodesic_horoball},
$\busemann^{x_0}_\lambda([\omega](x_0))\ge0$, while membership in the
closed horoball gives the reverse inequality. Thus this value is zero,
and \eqref{eq:busemann_ue} gives
$[\lambda]\in\pml(x_0,[\omega](x_0))$. 
Since $G$ is countable, it follows that
$\overline{\horos([\lambda];x_0)}\cap G(x_0)=\{x_0\}$ for
$\PThursM^{x_0}$-almost every
$[\lambda]\in \Dirichlet^{\infty}_G(x_0)$.
\end{proof}

\begin{lemma}[Dirichlet sets and Dirichlet polyhedra]
\label{lem:DirichletSets_Polyhedrons}
Let $G$ be a nontrivial subgroup of the Teichm\"uller modular group,
and let $x_0\in \teich_{g,m}$ satisfy
$\operatorname{Stab}_G(x_0)=\{\mathrm{id}\}$. Then the symmetric
difference
\begin{equation}
\label{eq:symmetric_differenice_Dirichletset_Dirichlet_polyhedron}
\operatorname{Dir}_G(x_0)\bigtriangleup \bigcup_{[\omega]\in G}[\omega](\Dirichlet^\infty_G(x_0))
\end{equation}
is $\PThursM^{x_0}$-null.
\end{lemma}

\begin{proof}
Let $E\subset\Dirichlet^\infty_G(x_0)$ consist of the uniquely
ergodic points $[\lambda]$ such that
\[
\overline{\horos([\lambda];x_0)}\cap G(x_0)=\{x_0\}.
\]
By the proof of \Cref{lem:almost-unique},
$[\omega](E)\cap E=\emptyset$ for every
$[\omega]\in G\setminus\{\mathrm{id}\}$, and
$\Dirichlet^\infty_G(x_0)\setminus E$ is $\PThursM^{x_0}$-null.
By definition, $E\subset \operatorname{Dir}_G(x_0)$. Furthermore,
\[
\operatorname{Dir}_G(x_0)
\setminus\bigcup_{[\omega]\in G}[\omega](E)
\subset
(\pml\setminus\pml^{ue})
\cup
\bigcup_{\substack{[\omega_1],[\omega_2]\in G\\
[\omega_1]\ne[\omega_2]}}
\pml([\omega_1](x_0),[\omega_2](x_0)).
\]
Indeed, outside the exceptional set on the right, every Dirichlet
point is uniquely ergodic and has a unique minimizing orbit point.
Applying \Cref{lem:Dirichlet_geodesic_horoball} at that point and
then translating back to $x_0$ shows that the Dirichlet point belongs
to $[\omega](E)$ for some $[\omega]\in G$. The exceptional set is
null by \Cref{thm:level_set} and the countability of $G$. The triviality
of $\operatorname{Stab}_G(x_0)$ ensures that
$[\omega_1](x_0)\ne[\omega_2](x_0)$ whenever
$[\omega_1]\ne[\omega_2]$.
Conversely, equivariance of extremal length shows that every point
of $[\omega](E)$ is uniquely ergodic and has $[\omega](x_0)$ as its
unique minimizing orbit point. Therefore,
\[
\bigcup_{[\omega]\in G}[\omega](E)\subset\operatorname{Dir}_G(x_0).
\]
Finally, replacing $E$ with $\Dirichlet^\infty_G(x_0)$ changes the
union only by a null set, since the action is nonsingular and $G$ is
countable. This proves the assertion.
\end{proof}

\subsection{Proof of \Cref{thm:Dissipative_part}}
\label{subsec:Proof-Dissipative_part}
Let $\dissipativepart^0_{x_0}(G)$ denote the subset of
$\dissipativepart_{x_0}(G)\cap\pml^{ue}$ consisting of points for
which the infimum defining $\horo^{x_0}_G([\lambda])$ is attained at
a unique point of the orbit $G(x_0)$.

We first show that $\dissipativepart^0_{x_0}(G)$ is conull in
$\dissipativepart_{x_0}(G)$. Suppose that
$[\lambda]\in\dissipativepart_{x_0}(G)\cap\pml^{ue}$. The defining
series then converges. A summable family of positive terms attains
its supremum, since it has only finitely many terms above any fixed
positive threshold. Hence $\horo^{x_0}_G([\lambda])>0$, and the
infimum defining it is attained. If two distinct orbit points
$y,z\in G(x_0)$ both attain the minimum, then $[\lambda]\in\pml(y,z)$.
Because the orbit is countable and $\pml^{ue}$ has full measure,
\Cref{thm:level_set} shows that the exceptional set is null.

Let $[\lambda]\in\dissipativepart^0_{x_0}(G)$, and choose
$[\omega]\in G$ so that the infimum defining
$\horo^{x_0}_G([\lambda])$ is attained at $[\omega](x_0)$. Set
$r_0=\busemann^{x_0}_\lambda([\omega](x_0))=\frac{1}{2}\log\horo^{x_0}_G([\lambda])$.
The uniqueness of the minimizing orbit point and
\eqref{eq:horoball-basepoint2} imply that
\[
G(x_0)\cap\overline{\horos([\lambda];[\omega](x_0))}
=G(x_0)\cap\overline{\horos([\lambda],r_0;x_0)}
=\{[\omega](x_0)\}.
\]
Applying \Cref{lem:Dirichlet_geodesic_horoball} with basepoint
$[\omega](x_0)$ gives
$[\lambda]\in\Dirichlet^\infty_G([\omega](x_0))$. Since
$\Dirichlet^\infty_G([\omega](x_0))=[\omega](\Dirichlet^\infty_G(x_0))$,
we conclude that
\[
\dissipativepart_{x_0}(G)
\setminus\bigcup_{[\omega]\in G}
[\omega](\Dirichlet^\infty_G(x_0))
\]
is $\PThursM^{x_0}$-null.

For the converse, let $E\subset\Dirichlet^\infty_G(x_0)$ consist of
the uniquely ergodic points $[\lambda]$ satisfying
$\overline{\horos([\lambda];x_0)}\cap G(x_0)=\{x_0\}$. By the proof
of \Cref{lem:almost-unique}, the sets $[\omega](E)$ and $E$ are
disjoint for every $[\omega]\in G\setminus\{\mathrm{id}\}$, and
$\Dirichlet^\infty_G(x_0)\setminus E$ is $\PThursM^{x_0}$-null.

The set $E$ is measurable. On $\pml^{ue}$, its defining condition is
\[
\ext_{[\omega](x_0)}(\lambda)>\ext_{x_0}(\lambda)
\quad\text{for every }[\omega]\in G\setminus\{\mathrm{id}\},
\]
a countable family of measurable conditions. We now show that
$E\setminus\dissipativepart_{x_0}(G)$ is null. By construction, the
sets $[\omega]^{-1}(E)$, for $[\omega]\in G$, are pairwise disjoint.
Tonelli's theorem and \eqref{eq:Radon-Nikodym2} therefore give
\begin{align*}
\int_E
\sum_{[\omega]\in G}
\left(
\frac{\ext_{x_0}(\lambda)}
{\ext_{[\omega](x_0)}(\lambda)}
\right)^{\convgenus}
d\PThursM^{x_0}([\lambda])
&=
\sum_{[\omega]\in G}
\int_E
\frac{d([\omega]_*\PThursM^{x_0})}{d\PThursM^{x_0}}([\lambda])
d\PThursM^{x_0}([\lambda])\\
&=
\sum_{[\omega]\in G}
\PThursM^{x_0}([\omega]^{-1}(E)) \\
&\le \PThursM^{x_0}(\pml)=1<\infty.
\end{align*}
Thus the defining series for $\dissipativepart_{x_0}(G)$ is finite
for $\PThursM^{x_0}$-almost every point of $E$, so
$E\setminus\dissipativepart_{x_0}(G)$ is null. Together with
\Cref{lem:almost-unique}, this shows that
$\Dirichlet^\infty_G(x_0)\setminus\dissipativepart_{x_0}(G)$ is null.
By \Cref{lem:Cons-Diss-ElementaryProperty},
$\dissipativepart_{x_0}(G)$ is $G$-invariant. Since the action is
nonsingular and $G$ is countable,
\[
\left(
\bigcup_{[\omega]\in G}
[\omega](\Dirichlet^\infty_G(x_0))
\right)
\setminus\dissipativepart_{x_0}(G)
\]
is also $\PThursM^{x_0}$-null. These two inclusions prove the
assertion. The symmetric difference
$\dissipativepart(G)\bigtriangleup \operatorname{Dir}_G(x_0)$ is null
by \Cref{lem:DirichletSets_Polyhedrons}.

Mutual absolute continuity of the measures
$\{\PThursM^x\}_{x\in\teich_{g,m}}$ gives the result for every
$x\in\teich_{g,m}$. Together with the wandering property established in
\Cref{lem:almost-unique}, this proves all the assertions concerning the
dissipative part in \Cref{thm:main_classfication_limit_sets}.
\qed

\subsection{Proof of \Cref{coro:conservativity_Dirichletset}}
By \Cref{thm:Hopf_decomposition_main}, (a) and (c) are equivalent.
The first assertion of \Cref{thm:main_classfication_limit_sets}
gives the equivalence with (b), and
\Cref{thm:Dissipative_part} gives the equivalence with (d).
Because the action is nonsingular, a countable union of translates of
$\Dirichlet^\infty_G(x_0)$ is null exactly when
$\Dirichlet^\infty_G(x_0)$ is null. A further application of
\Cref{thm:Dissipative_part} gives the equivalence with (e). \qed

\section{Conical limit sets of the Torelli group}
\label{sec:dynamics_of_Torelli_group}

Throughout this section, we write $\teich_g=\teich_{g,0}$ and $\pml=\pml_{g,0}$. We prove that the conical limit set of the Torelli group of a closed surface of genus $g\ge2$ has Thurston measure zero. We use the period map to compare the radial limits of typical Teichm\"uller rays with those of rays that return within bounded distance of a Torelli orbit. We first recall the relevant properties of Siegel space and the period map, and then establish the radial-limit statements needed for this comparison.

\subsection{Siegel upper half-space}
Following Friedland and Freitas \cite{FriedlandFreitas2004}, we recall the Siegel upper half-space, its realization as the Siegel disk, and the Cayley transform relating their boundaries. These domains are classical objects in the theory of Hermitian symmetric spaces; see, for example, Helgason \cite{Helgason2001} and Satake \cite{Satake1980}.

Let $\mat_{m,n}(\mathbb{C})$ denote the set of $m\times n$ complex matrices, and let
$\symm(g,\mathbb{C})\subset \mat_{g,g}(\mathbb{C})$ denote the set of symmetric matrices.
The \emph{Siegel upper half-space} $\mathfrak{H}_g$ of degree $g$ is the open subset of
$\symm(g,\mathbb{C})$ consisting of symmetric matrices whose imaginary parts are positive definite.
The symplectic group
\[
\spl(2g,\mathbb{R})=\left\{
M\in \gl(2g,\mathbb{R})\mid M^TJ_gM=J_g\right\}
\quad
\left(
J_g=\begin{pmatrix} 0 & I_g \\ -I_g & 0\end{pmatrix}
\right)
\]
acts on $\mathfrak{H}_g$ by
\[
M(Z)=(AZ-B)(-CZ+D)^{-1}
\quad
\left(M=\begin{pmatrix} A & B \\ C & D\end{pmatrix}\right).
\]
This convention is compatible with the action of the symplectic group on the first homology group of $\Sigma_g$
(cf. \S\ref{subsec:symplectic_representation} and \eqref{eq:symplectic_rep}; see also \cite[(6.8)]{Earle1978}).

Let
\[
\mathfrak{D}_g=\{W\in \symm(g,\mathbb{C})\mid I_g-W\overline{W}>0\}.
\]
This domain is called the \emph{Siegel disk}, or the \emph{generalized unit disk}, of degree $g$
(cf. \cite{Siegel1943}).
The \emph{Cayley transform}
\begin{equation}
\label{eq:cayray}
\Cayley\colon\mathfrak{H}_g\ni Z\mapsto (Z-\sqrt{-1}I_g)(Z+\sqrt{-1}I_g)^{-1}\in \mathfrak{D}_g
\end{equation}
is biholomorphic. Via this transform, $\spl(2g,\mathbb{R})$ acts naturally on
$\mathfrak{D}_g$, and this action extends to the boundary $\partial \mathfrak{D}_g$ in
$\symm(g,\mathbb{C})$.
The following criterion describes convergence to a finite boundary point of the Siegel upper half-space (cf. Friedland and Freitas \cite{FriedlandFreitas2004}).

\begin{proposition}
\label{prop:F-F-lemma3.1}
A sequence $Z_n\in\mathfrak H_g$ converges to a point of
$\partial\mathfrak H_g$ in the ambient matrix space if and only if
$\Cayley(Z_n)$ converges to a point $W_\infty\in\partial\mathfrak D_g$
with $\det(I_g-W_\infty)\ne0$.
\end{proposition}

\subsection{Symplectic representation and the Torelli group}
\label{subsec:symplectic_representation}
The action of the mapping class group on $H_1(\Sigma_g,\mathbb{Z})$ defines a homomorphism
\[
\hat{\symrep}\colon \mcg(\Sigma_g)\to
{\rm Aut}(H_1(\Sigma_g,\mathbb Z)),
\]
called the \emph{symplectic representation} (cf. \cite[Chapter Six]{FarbMargalit2012}). Its kernel is the \emph{Torelli group} $\torelliG_g$. Fix a symplectic homology basis on $\Sigma_g$, written as
\[
\{\mathbf a,\mathbf b\}
=\{\mathbf a_1,\ldots,\mathbf a_g,\mathbf b_1,\ldots,\mathbf b_g\}.
\] Relative to this basis, the symplectic representation is given by
\begin{equation}
\label{eq:symp_homo1}
\hat{\symrep}\colon \mcg(\Sigma_{g})\to \syp(2g,\mathbb{Z}).
\end{equation}
Writing $H_1(\Sigma_g,\mathbb{Z})=\bigoplus_{j=1}^g(\mathbb{Z}\mathbf{a}_j\oplus \mathbb{Z}\mathbf{b}_j)$, the action on homology is given by
 \[
 (\omega(\mathbf{a}),\omega(\mathbf{b}))=(\mathbf{a},\mathbf{b}) \hat{\symrep}([\omega]_*)
 \]
 for $[\omega]_*\in \mcg(\Sigma_{g})$. Composing \eqref{eq:symp_homo1} with the natural projection $\syp(2g,\mathbb Z)\to\psyp(2g,\mathbb Z)$ yields a homomorphism
\begin{equation}
\label{eq:symp_homo2}
\symrep\colon \mcg(\Sigma_{g})\to \psyp(2g,\mathbb{Z})\subset \operatorname{Aut}(\mathfrak{H}_g).
\end{equation}

\subsection{Period map}
\label{subsec:period_map}
We use the symplectic basis fixed in \S\ref{subsec:symplectic_representation}.
For $x=(M,f)\in \teich_{g}$, let $\omega_x^j$ be the holomorphic
$1$-form on $M$ satisfying
\[
\int_{f_*(\mathbf{a}_k)}\omega_x^j=\delta_{jk},
\]
where $\delta_{jk}$ is the Kronecker delta. Define the \emph{period
matrix} $\Pi(x)=\begin{pmatrix} \Pi_{jk}(x)\end{pmatrix}$ of $x$ relative to this basis by
\[
\Pi_{jk}(x)=\int_{f_*(\mathbf{b}_k)}\omega_x^j
\]
(cf. \cite{FarkasKra1992}).
The resulting \emph{period map}
$\Pi\colon\teich_g\to\mathfrak H_g$
is holomorphic \cite{Ahlfors1960} and satisfies the equivariance relation
\begin{equation}
\label{eq:symplectic_rep}
\symrep([\omega]_*)(\Pi(x))=\Pi\circ [\omega]_*(x)
\end{equation}
for $x\in \teich_{g}$ and $[\omega]\in \tmod(g)$. In particular, the period map is invariant under the Torelli group.

\subsection{Radial limits of the period map}
The bounded realization $\Cayley\circ\Pi$ admits radial limits almost everywhere. The following lemma shows that, after removing a null set, the inverse Cayley transform gives finite radial limits of $\Pi$ that are independent of the basepoint.
\begin{lemma}[Radial limit of the period map]
    \label{lem:radial_limit}
There is a measurable set $\mathcal{E}_1\subset\pml$ of full Thurston measure with the following properties:
\begin{itemize}
    \item[(a)] $\mathcal E_1$ consists of projective classes of uniquely
ergodic measured laminations and is invariant under the action of the Teichm\"uller modular group
    $\tmod(g)$;
    \item[(b)] there is a measurable map
    \[
    \Pi^*\colon \mathcal{E}_1\to \overline{\mathfrak{H}_g}
    \subset \symm(g,\mathbb{C})
    \]
    such that, for any $[\lambda]\in \mathcal{E}_1$ and any $x\in \teich_{g}$,
    \[
    \lim_{t\to \infty}\Pi(\ray^{\lambda}_x(t))=\Pi^*([\lambda]).
    \]
    In particular, this limit is independent of the choice of $x$;
    \item[(c)] for every $[\lambda]\in \mathcal{E}_1$ and every
    $\begin{bmatrix} A & B \\ C & D\end{bmatrix}\in \syp(2g,\mathbb{Z})$, the radial limit satisfies
    \[
    \det(-C\Pi^*([\lambda])+D)\ne 0.
    \]
\end{itemize}
\end{lemma}

\begin{remark}
By \Cref{lem:null_siegel} below, the radial limit $\Pi^*([\lambda])$ lies in the boundary $\partial\mathfrak{H}_g$ for almost every $[\lambda]\in \mathcal{E}_1$.
\end{remark}

\begin{proof}[Proof of \Cref{lem:radial_limit}]
By \cite[Theorem 3.1]{Miyachi2024Bounded},
there exist a $\tmod(g)$-invariant full-measure set
$\mathcal{E}_0\subset \mathcal{PML}^{ue}$ and a measurable map
\[
(\Cayley\circ\Pi)^*\colon
\mathcal E_0\to\overline{\mathfrak D_g}
\]
such that the following hold:
\begin{enumerate}
\item[(1)] for any $x\in \teich_{g}$ and $[\lambda]\in \mathcal{E}_0$,
\[
\lim_{t\to \infty}(\Cayley\circ\Pi)(\ray^{\lambda}_x(t))
=
(\Cayley\circ\Pi)^*([\lambda]);
\]
\item[(2)] for any $[\omega]_*\in \tmod(g)$,
\[
\symrep([\omega]_*)\bigl((\Cayley\circ\Pi)^*([\lambda])\bigr)
=
(\Cayley\circ\Pi)^*([\omega]_*[\lambda])
\]
for all $[\lambda]\in \mathcal{E}_0$.
\end{enumerate}
The set $\mathcal E_0$ therefore satisfies condition (a).

For
$M=\begin{pmatrix} A & B \\ C & D\end{pmatrix}\in \syp(2g,\mathbb{Z})$, define
\[
f_M(x)
=
\det\bigl(-\sqrt{-1}C+D-(\sqrt{-1}C+D)\Cayley(\Pi(x))\bigr).
\]
For the identity matrix, this gives
\[
f_{I_{2g}}(x)=\det(I_g-\Cayley(\Pi(x))).
\]
Each $f_M$ is bounded and holomorphic on $\teich_{g}$ because
$\Cayley(\Pi(x))\in \mathfrak{D}_g$. Writing
$W(x)=\Cayley(\Pi(x))$, we have
\[
f_M(x)
=
\det(-C\Pi(x)+D)\det(I_g-W(x))
\ne0
\]
for every $x\in\teich_g$.
Indeed, the first factor is nonzero because the
symplectic action is defined on $\mathfrak H_g$,
and the second is nonzero because
$W(x)\in\mathfrak D_g$.

It follows from (1) that the radial limit of $f_M$ is
\[
f_M^*([\lambda])
=
\det\bigl(-\sqrt{-1}C+D-(\sqrt{-1}C+D)(\Cayley\circ \Pi)^*([\lambda])\bigr)
\]
for $[\lambda]\in \mathcal{E}_0$.
Since $f_M$ is not identically zero, the identity theorem for radial limit functions \cite[Theorem 1.2]{Miyachi2024Bounded} implies that the set
\[
\mathcal{E}_1^M
=
\{[\lambda]\in \mathcal{E}_0\mid f_M^*([\lambda])\ne 0\}
\]
has full measure in $\pml$.

Since $\syp(2g,\mathbb{Z})$ is countable, the intersection
\[
\mathcal{E}_1
=
\bigcap_{M\in \syp(2g,\mathbb{Z})}\mathcal{E}_1^M
\]
also has full measure in $\pml$. Its invariance under $\tmod(g)$ follows from (2) and the composition law for the symplectic action.
In particular, since $f_{I_{2g}}^*([\lambda])\ne 0$ on $\mathcal{E}_1$, we have
\[
\det\bigl(I_g-(\Cayley\circ\Pi)^*([\lambda])\bigr)\ne 0
\]
for all $[\lambda]\in \mathcal{E}_1$.
For $[\lambda]\in\mathcal{E}_1$, set
\[
W_\lambda=(\Cayley\circ\Pi)^*([\lambda]).
\]
Since $\det(I_g-W_\lambda)\ne0$, the map
\[
W\longmapsto
\sqrt{-1}(I_g+W)(I_g-W)^{-1}
\]
is well-defined and continuous in a neighborhood of $W_\lambda$.
By the inverse Cayley transform formula, for every
$x\in\teich_g$ we therefore have
\[
\lim_{t\to\infty}\Pi(\ray^\lambda_x(t))
=
\sqrt{-1}(I_g+W_\lambda)(I_g-W_\lambda)^{-1}.
\]
Define
\[
\Pi^*([\lambda])
=
\sqrt{-1}(I_g+W_\lambda)(I_g-W_\lambda)^{-1}.
\]
This matrix lies in $\overline{\mathfrak H_g}$ as a limit of matrices
in $\mathfrak H_g$. The map $[\lambda]\mapsto W_\lambda$ is measurable
and independent of the basepoint, so the same holds for $\Pi^*$.
Thus, $\mathcal E_1$ satisfies condition (b).
Since the non-vanishing of $f_M^*([\lambda])$ is equivalent to
\[
\det(-C\Pi^*([\lambda])+D)\ne 0,
\]
$\mathcal{E}_1$ satisfies condition (c).

\end{proof}

\subsection{Proof of the nullity of the conical limit set}
We compare the radial behavior of the period map with the return of a ray to a bounded neighborhood of a Torelli orbit. \Cref{lem:null_siegel} shows that almost every radial limit lies in $\partial\mathfrak H_g$. In contrast, \Cref{lem:conical_limit_set_torelli_group} shows that the radial limit of every direction in $\mathcal E_1\cap\Lambda_C(\torelliG_g)$ lies in $\mathfrak H_g$. Together, these statements imply \Cref{cor:conical_limit_sets}.

For $x\in\teich_g$, let $\PSiegelThursM^x$ be the pushforward of $\PThursM^x$ to $\overline{\mathfrak{H}_g}$ under $\Pi^*$:
\[
\PSiegelThursM^x=(\Pi^*)_*(\PThursM^x).
\]

\begin{lemma}
\label{lem:null_siegel}
For any $x\in \teich_{g}$,
$\PSiegelThursM^x\bigl(\mathfrak{H}_g\bigr)=0$.
\end{lemma}

\begin{proof}
Fix $x_0\in\teich_g$.
We first show that $\PSiegelThursM^{x_0}(\{Z\})=0$ for every
$Z\in\mathfrak H_g$.
Since the period map is nonconstant and the Cayley transform is biholomorphic, some component $h=(\Cayley\circ\Pi)_{ij}$ is a nonconstant bounded holomorphic function on $\teich_g$.
If $\PSiegelThursM^{x_0}(\{Z\})>0$, then the radial limit of $h$ equals
the constant $(\Cayley(Z))_{ij}$ on a set of positive
Thurston measure.
By the identity theorem for radial limits
\cite[Theorem 1.2]{Miyachi2024Bounded}, this would imply
that $h$ is constant, a contradiction.

Suppose now that $\PSiegelThursM^{x_0}(\mathfrak H_g)>0$, and set
\[
E=\{[\lambda]\in\mathcal E_1
       \mid \Pi^*([\lambda])\in\mathfrak H_g\}.
\]
The equivariance of $\Pi^*$ implies that $E$ is invariant
under the mapping class group.
By ergodicity, $\PThursM^{x_0}(E)=1$.

Since $\psyp(2g,\mathbb{Z})$ acts properly on $\mathfrak{H}_g$, its orbit space is locally compact, second countable, and Hausdorff, and hence is a standard Borel space.
Let
\[
q\colon\mathfrak{H}_g\longrightarrow
\psyp(2g,\mathbb{Z})\backslash\mathfrak{H}_g
\]
be the quotient map.
The measurable map $q\circ\Pi^*$ on $E$ is invariant
under the mapping class group.
By ergodicity, the pushforward of $\PThursM^{x_0}|_E$ under $q\circ\Pi^*$ assigns either zero or one to every Borel set. A probability measure with this property on a standard Borel space is a Dirac measure. Thus $q\circ\Pi^*$ is essentially constant, and there exists $Z_0\in\mathfrak{H}_g$ such that
\[
\PThursM^{x_0}\bigl(\{[\lambda]\in E
          \mid \Pi^*([\lambda])\in \mathcal{O}\}\bigr)=1,
\]
where $\mathcal{O}$ is the $\psyp(2g,\mathbb{Z})$-orbit of $Z_0$.
However, the orbit $\mathcal{O}$ is countable, and the first part of the proof shows that
\[
\PSiegelThursM^{x_0}(\mathcal{O})
\le
\sum_{A\in\psyp(2g,\mathbb{Z})}\PSiegelThursM^{x_0}(\{A(Z_0)\})
=0.
\]
This is a contradiction.
Thus $\PSiegelThursM^{x_0}(\mathfrak H_g)=0$.
Since $x_0\in \teich_g$ was arbitrary, the lemma follows.
\end{proof}

\begin{lemma}
    \label{lem:conical_limit_set_torelli_group}
If $[\lambda]\in\mathcal{E}_1\cap\Lambda_C(\torelliG_g)$, then $\Pi^*([\lambda])\in\mathfrak{H}_g$.
\end{lemma}

\begin{proof}
There exist a point $x_1\in \teich_{g}$, possibly different from the fixed basepoint $x_0$, a sequence of distinct elements $\{\varphi_n\}_{n=1}^\infty$ in $\torelliG_{g}$, a sequence $\{t_n\}_{n=1}^\infty$ in $[0,\infty)$ with $t_n\to \infty$, and a constant $D>0$ such that
\[
d_T(\ray^{\lambda}_{x_1}(t_n),\varphi_n(x_0))\le D
\]
for all $n\ge 1$.
The period map is holomorphic and invariant under the Torelli group. Since the Teichm\"uller distance agrees with the Kobayashi distance, its distance-decreasing property gives
\[
d_{\mathfrak{H}_g}(\Pi(\ray^{\lambda}_{x_1}(t_n)),\Pi(x_0))
=
d_{\mathfrak{H}_g}(\Pi(\ray^{\lambda}_{x_1}(t_n)),\Pi(\varphi_n(x_0)))
\le D
\]
for all $n\ge 1$.
Since $\mathfrak{H}_g$ is a proper complete metric space with respect to its invariant Kobayashi distance, the sequence
$\{\Pi(\ray^{\lambda}_{x_1}(t_n))\}_{n=1}^{\infty}$
has a subsequence converging to a point of $\mathfrak{H}_g$.
By \Cref{lem:radial_limit}(b), the entire sequence converges to $\Pi^*([\lambda])$. Its subsequential limit in $\mathfrak{H}_g$ must therefore equal $\Pi^*([\lambda])$, proving the assertion.
\end{proof}

The preceding lemmas imply the following corollary.

\begin{corollary}[The conical limit set has measure zero]
\label{cor:conical_limit_sets}
The conical limit set $\Lambda_C(\torelliG_{g})$ of the Torelli group has measure zero with respect to the Thurston measure.
\end{corollary}

\begin{proof}
By \Cref{lem:conical_limit_set_torelli_group},
$\mathcal{E}_1\cap\Lambda_C(\torelliG_g)$ is contained in
$(\Pi^*)^{-1}(\mathfrak{H}_g)$, which is null by \Cref{lem:null_siegel}.
Since $\mathcal{E}_1$ has full measure, the corollary follows.
\end{proof}

In contrast, the conical limit set of the Teichm\"uller modular group
$\tmod(g)$ has full measure. This follows from the conservativity of
the Teichm\"uller geodesic flow and the recurrence of almost every
Teichm\"uller geodesic ray (cf.\ \cite{Masur1982}). The fixed points of every pseudo-Anosov element of a subgroup of
$\tmod(g)$ belong to the conical limit set of that subgroup. This comparison concerns Thurston
measure on the boundary. For results on the genericity of
pseudo-Anosov elements, see, for example,
\cite{ChoiInhyeok2024, VivekaSoutoJing2020, Maher2011, MalesteinSouto2013, Rivin2008, Rivin2009}.

The nullity of the conical limit set also has the following consequence for recurrence in the Torelli quotient.

\begin{corollary}
With respect to the measure induced by the Masur--Veech measure,
almost every trajectory of the Teichm\"uller geodesic flow
on the unit cotangent bundle of the Torelli space
$\teich_g/\torelliG_g$ is non-recurrent in forward time.
\end{corollary}

\begin{proof}
If a forward trajectory projects back to a compact subset of the
Torelli space infinitely often, its vertical projective foliation
belongs to $\Lambda_C(\torelliG_g)$. Indeed, compactness in the
quotient implies that, along a sequence of times tending to infinity,
a lift of the trajectory stays within bounded Teichm\"uller distance
of a fixed $\torelliG_g$-orbit.

By \Cref{cor:conical_limit_sets}, these directions form a
Thurston-null set. Disintegration of the Masur--Veech measure over
Teichm\"uller space, together with the fact that the conditional
measures on unit cotangent spheres belong to the Thurston measure class
\cite[\S2.3, Proposition 2.3(i), and \S2.3.1]{Mirzakhani_etal2012}, shows
that the corresponding initial conditions form a Masur--Veech-null set. The same is true on the quotient by $\torelliG_g$, proving the
claim.
\end{proof}

\subsection{Complement of the small horospherical limit set}
We conclude by exhibiting directions outside the small horospherical limit set of the Torelli group. The construction uses Teichm\"uller rays defined by squares of holomorphic $1$-forms.

\begin{proposition}[$\pml\setminus \Lambda_h(\torelliG_g)\ne \emptyset$]
\label{prop:complement_small_horospherical_limit_set}
Let $g\ge 2$ and $x_0=(M_0,f_0)\in\teich_g$. If $[\lambda]\in\pml$ is represented by the vertical measured foliation of a nonzero holomorphic $1$-form on $M_0$, then $[\lambda]\in\operatorname{Dir}_{\torelliG_g}(x_0)\setminus\Lambda_h(\torelliG_g)$.
\end{proposition}

\begin{proof}
Choose a representative $\lambda$ of $[\lambda]$ and a holomorphic $1$-form $\eta$ on $M_0$ such that $v(\eta^2)=\lambda$. By \cite{Kra1981}, the map
$\Pi\circ \ray^\lambda_{x_0}\colon [0,\infty)\to \mathfrak{H}_g$
is an isometric embedding from $[0,\infty)$ with its usual metric
into $\mathfrak{H}_g$ with its Kobayashi metric. The distance-decreasing property of the period map and its invariance under the Torelli group imply that, for any $[\omega]\in\torelliG_g$ and $t\ge0$,
\begin{align*}
d_T([\omega](x_0),\ray^\lambda_{x_0}(t))
&\ge d_{\mathfrak{H}_g}(\Pi\circ [\omega](x_0),\Pi\circ \ray^\lambda_{x_0}(t)) \\
&\ge d_{\mathfrak{H}_g}(\Pi(x_0),\Pi\circ \ray^\lambda_{x_0}(t)) \\
&= d_{\mathfrak{H}_g}(\Pi\circ \ray^\lambda_{x_0}(0),\Pi\circ \ray^\lambda_{x_0}(t))=t.
\end{align*}
Subtracting $t$ and letting $t\to\infty$, we obtain
$\busemann^{x_0}_\lambda([\omega](x_0))\ge0$ for every
$[\omega]\in\torelliG_g$. Hence $[\lambda]\notin\Lambda_h(\torelliG_g)$.
Since $\busemann^{x_0}_\lambda(x_0)=0$, we also have
$[\lambda]\in\operatorname{Dir}_{\torelliG_g}(x_0)$.
\end{proof}

\section{Appendix: Hopf--Tsuji--Sullivan theorem for subgroups}
\label{sec:Hopf-Tsuji-Sullivan}
This appendix proves a subgroup version of the Hopf--Tsuji--Sullivan theorem for $\tmod(g,m)$ for all surface types satisfying our standing assumptions. It complements the closed-surface formulation discussed by Yang \cite{yang2025conformaldynamicsinfinitygroups} and supplies the characterization used in the introduction.

The characterization relates the measure of the conical limit set to the conservativity and ergodicity of the geodesic flow and to the ergodicity of the diagonal boundary action.

\begin{theorem}[Hopf--Tsuji--Sullivan-type characterization]
\label{thm:Hopf-Tsuji-Sullivan}
Let $G$ be a subgroup of $\tmod(g,m)$ and let $x\in\teich_{g,m}$. Write $\Lambda_C=\Lambda_C(G)$. The following conditions are equivalent:
\begin{itemize}
\item[(a)] $\PThursM^x(\Lambda_C)>0$;
\item[(b)] $\PThursM^x(\Lambda_C)=1$;
\item[(c)] the Teichm\"uller geodesic flow $\Phi_G$ on $\mathcal{Q}^1_{g,m}/G$ is conservative with respect to the invariant measure $\nu_G$ defined below;
\item[(d)] this flow is ergodic with respect to $\nu_G$; and
\item[(e)] the diagonal action of $G$ on $(\pml_{g,m}^2,\PThursM^{x;\otimes2})$ is ergodic.
\end{itemize}
\end{theorem}
We prove the implications in the order
$(a)\Rightarrow(b)\Rightarrow(c)\Rightarrow(d)\Rightarrow(e)\Rightarrow(a)$,
adapting Roblin's argument for the Hopf--Tsuji--Sullivan dichotomy in the
$\operatorname{CAT}(-1)$ setting \cite{Roblin2003}. Recurrence first yields a zero--one law for the conical limit set. The Hopf argument then gives ergodicity of the flow, while the final implication uses a measurable wandering domain for pairs of nonconical endpoints.

\subsection*{(a) $\Rightarrow$ (b)}
We begin with the following general lemma.

\begin{lemma}
\label{lem:general_theory1}
Let $X$ be a locally compact, $\sigma$-compact, second-countable
Hausdorff space, and let
$\mathbb{R}\times X\ni(t,x)\mapsto\phi^t(x)\in X$ be a continuous flow
preserving a Radon measure $m$ on $X$.
Let $R_+$ (respectively, $R_-$) be the set of points whose orbits meet some compact subset of $X$ along a sequence of times tending to $+\infty$ (respectively, $-\infty$). Then $m(R_+\mathbin{\triangle}R_-)=0$.
Moreover, both $R_+$ and $R_-$ coincide, up to null sets, with the
conservative part of the time-one map $\phi^1\colon X\to X$.
\end{lemma}

\begin{proof}
Let $X=\mathscr{C}\cup\mathscr{D}$ be the Hopf decomposition of the
time-one map $T=\phi^1$.
There is a wandering set $W$ such that
$\mathscr{D}=\bigcup_{n\in\mathbb{Z}}T^n(W)$ up to null sets
(cf.\ \cite[Theorem~3.2]{Krengel1985}). Thus $\mathscr{D}$ is invariant under both $T$ and $T^{-1}$, and the same decomposition $X=\mathscr{C}\cup\mathscr{D}$ is the Hopf decomposition of $T^{-1}=\phi^{-1}$.

Choose a compact exhaustion $\{X_n\}_{n=1}^\infty$ such that
$X_n\subset\operatorname{int}(X_{n+1})$. By conservativity, almost every point of
$X_n\cap\mathscr{C}$ returns to $X_n\cap\mathscr{C}$ infinitely often
under both positive and negative iterates of $T$
(cf.\ \cite[Theorem~3.1]{Krengel1985}). Consequently,
$\mathscr{C}\subset R_+\cap R_-$ up to a null set.

By the Fubini--Tonelli theorem,
\[
\int_W\sum_{j\in\mathbb{Z}}\boldsymbol{1}_{X_n}(T^j(x))\,dm(x)
=\sum_{j\in\mathbb{Z}}m(X_n\cap T^j(W))
=m(X_n\cap\mathscr{D})<\infty.
\]
Thus, after discarding a null subset of $W$, the orbit of each remaining point visits every $X_n$ only finitely many times.
Let $K$ be a compact set in $X$.
Suppose that, for some $x\in W$, there is an increasing sequence
$t_j\to\infty$ such that $\phi^{t_j}(x)\in K$. Then there is a sequence
$k_j\in\mathbb{Z}$ with $k_j\to\infty$ such that
$\phi^{k_j}(x)\in\bigcup_{-1\le t\le0}\phi^t(K)$ for every $j$.
The latter set is compact and hence is contained in $X_n$ for some $n$,
contradicting the preceding conclusion. Applying the same argument to a
decreasing sequence $t_j\to-\infty$, we obtain
\[
m\bigl(\mathscr{D}\cap(R_+\cup R_-)\bigr)
\le \sum_{n\in\mathbb{Z}}m(T^n(W)\cap R_+)
+\sum_{n\in\mathbb{Z}}m(T^n(W)\cap R_-)=0.
\]
It follows that $m(R_+\mathbin{\triangle}R_-)=0$ and that both $R_+$
and $R_-$ agree with $\mathscr{C}$ up to null sets.
\end{proof}

We next recall some basic facts about the Teichm\"uller geodesic flow.

The holomorphic cotangent bundle of $\teich_{g,m}$ is naturally identified
with the vector bundle $\mathcal{Q}_{g,m}$ of holomorphic quadratic
differentials of finite $L^1$-norm. We equip each fiber with this norm. Let $\mathfrak{p}\colon \mathcal{Q}_{g,m}\to
\teich_{g,m}$ denote the bundle projection.

The Hubbard--Masur theorem \cite{HubbardMasur1979} identifies the unit cotangent bundle $\mathcal{Q}^1_{g,m}$ with the space of binding pairs $(\lambda_h,\lambda_v)\in
\ml_{g,m}\times\ml_{g,m}$ satisfying $i(\lambda_h,\lambda_v)=1$, where
$\lambda_h$ and $\lambda_v$ are the horizontal and vertical measured
laminations, respectively.
In these coordinates, the Teichm\"uller geodesic flow $\Phi^t$ is given by
\[
\ml_{g,m}\times \ml_{g,m}\ni (\lambda_h,\lambda_v)\mapsto (e^{-t}\lambda_h,e^t\lambda_v)\in \ml_{g,m}\times \ml_{g,m}
\]
and descends to a flow $\Phi_G^t$ on
$\mathcal{Q}^1_{g,m}/G$ (cf.\ \cite{Masur1982}; see also
\cite{Hamenstadt2009}).

We use the natural homeomorphism
\[
\mathcal{Q}^1_{g,m}\ni q\mapsto (\lambda_h(q),[\lambda_v(q)])
\in (\ml_{g,m}\times \pml_{g,m})\setminus \widetilde{\Delta}
\]
where
\[
\widetilde{\Delta}
=\bigcup_{[\eta]\in \pml_{g,m}}
\{(\lambda,[\nu])\mid i(\lambda,\eta)+i(\nu,\eta)=0\}.
\]
Following \cite[\S2.3]{Mirzakhani_etal2012}, we define a measure
\begin{equation}
\label{eq:PS-measure}
d\PSmeasure([\lambda],[\nu])
=\left(\frac{\ext_x(\lambda)\ext_x(\nu)}{i(\lambda,\nu)^2}\right)^{\convgenus}
d(\PThursM^x\otimes\PThursM^x)([\lambda],[\nu])
\end{equation}
on $(\pml_{g,m}^{ue}\times\pml_{g,m}^{ue})\setminus\Delta^{ue}$,
where $\Delta^{ue}$ denotes the diagonal. By
\eqref{eq:Radon-Nikodym}, this measure is independent of $x$ and invariant
under the diagonal action of the mapping class group. Projectivizing both endpoint laminations gives the projection
\[
\mathcal{P}\colon
\mathcal{Q}^1_{g,m}\ni q\mapsto ([\lambda_h(q)],[\lambda_v(q)])
\in (\pml_{g,m}\times \pml_{g,m})\setminus \Delta'
\]
where
\[
\Delta'=\bigcup_{[\eta]\in \pml_{g,m}}\{([\lambda],[\nu])\mid i(\lambda,\eta)+i(\nu,\eta)=0\}
\]
and $\lambda_h(q)$ and $\lambda_v(q)$ are the horizontal and vertical
laminations of $q$, respectively.

The projection $\mathcal{P}$ is flow-invariant, so
$\mathcal{P}\circ\Phi^t=\mathcal{P}$ on $\mathcal{Q}^1_{g,m}$. Moreover,
each point of $\mathcal{Q}^1_{g,m}$ has a neighborhood of the form
\begin{equation}
\label{eq:product-coordinatesU1U2-I}
U_1\times U_2\times I,
\end{equation}
where each $U_i$ is open in $\pml_{g,m}$ and $I\subset\mathbb{R}$ is an
open interval. In these coordinates, $\mathcal{P}$ is the projection onto
the first two factors.

Since
$\Delta'\cap(\pml_{g,m}^{ue}\times\pml_{g,m}^{ue})=\Delta^{ue}$ and
$\pml_{g,m}^{ue}$ has full measure in $\pml_{g,m}$, the local product measure
\begin{equation}
\label{eq:measure_nu}
d\nu=d\PSmeasure\otimes dt
\end{equation}
defines a $\Phi^t$-invariant measure on $\mathcal{Q}^1_{g,m}$ that is also
invariant under the Teichm\"uller modular group. It therefore descends to
an invariant measure $\nu_G$ on $\mathcal{Q}^1_{g,m}/G$
(cf.\ \cite{Mirzakhani_etal2012,Hamenstadt2009,LindenstraussMirzakhani2008}).

\begin{lemma}[The conical limit set is null or conull]
\label{lem:positivity_of_LC_implies_full}
Let $G$ be a subgroup of the Teichm\"uller modular group. For every
$x\in\teich_{g,m}$, if $\PThursM^x(\Lambda_C)>0$, then
$\PThursM^x(\Lambda_C)=1$.
\end{lemma}

\begin{proof}
Define $R_+,R_-\subset\mathcal{Q}^1_{g,m}/G$ as in
\Cref{lem:general_theory1}, with respect to the flow $\Phi_G$.
Let $P_G\colon\mathcal{Q}^1_{g,m}\to\mathcal{Q}^1_{g,m}/G$ be the
quotient map. For $\nu$-almost every $q\in\mathcal{Q}^1_{g,m}$, both
$[\lambda_h(q)]$ and $[\lambda_v(q)]$ are uniquely ergodic. For such a
point $q$, we claim that $P_G(q)\in R_+$ (respectively,
$P_G(q)\in R_-$) if and only if $[\lambda_v(q)]\in\Lambda_C$
(respectively, $[\lambda_h(q)]\in\Lambda_C$). Indeed, Teichm\"uller
geodesic rays with the same uniquely ergodic endpoint are asymptotic, so
conicality can be tested using the ray determined by $q$
\cite{Masur1980}. We prove the forward-time assertion; the backward-time
case is analogous.

Fix $x_0\in\teich_{g,m}$ and suppose that $P_G(q)\in R_+$. By
definition, there are a compact set $K\subset\mathcal{Q}^1_{g,m}/G$ and
an increasing sequence $t_k\to\infty$ such that
$\Phi_G^{t_k}(P_G(q))\in K$.
Set $q_k=\Phi^{t_k}(q)$ and $x_k=\mathfrak{p}(q_k)$. Thus $q_k$
corresponds to the pair
$(e^{-t_k}\lambda_h(q),e^{t_k}\lambda_v(q))$, and $x_k$ lies on the
Teichm\"uller geodesic ray with direction $[\lambda_v(q)]$.
Since
\[
\Phi_G^{t_k}(P_G(q))=P_G(\Phi^{t_k}(q)),
\]
compactness of $K$ implies that there exist $R>0$ and $g_k\in G$ such
that
\[
d_T(g_k(x_0),x_k)<R
\]
for every $k$. Hence
$[\lambda_v(q)]\in\Lambda_C$. Conversely, if $[\lambda_v(q)]\in\Lambda_C$, there are times $t_k\to\infty$ for which $x_k$ stays within bounded distance of $G(x_0)$. After applying suitable elements of $G$, the points $q_k$ lie in the unit cotangent bundle over a fixed closed Teichm\"uller ball. This bundle is compact, so their images in the quotient return to a compact set.

By \Cref{lem:general_theory1}, $\nu_G(R_+\setminus R_-)=0$.
Consequently, in every coordinate neighborhood
$U_1\times U_2\times I$ as in
\eqref{eq:product-coordinatesU1U2-I},
\[
\PSmeasure\left(((\pml_{g,m}\setminus\Lambda_C)\cap U_1)
\times(\Lambda_C\cap U_2)\right)\operatorname{Leb}(I)=0,
\]
where $\operatorname{Leb}$ denotes Lebesgue measure on $\mathbb{R}$. A countable cover by these coordinate neighborhoods shows that
$(\pml_{g,m}\setminus\Lambda_C)\times\Lambda_C$ is $\PSmeasure$-null.
Since $\PSmeasure$ and the product measure
$\PThursM^{x;\otimes 2}$ are mutually absolutely continuous, it follows
that
\[
\PThursM^{x}(\pml_{g,m}\setminus \Lambda_C)\PThursM^x(\Lambda_C)=
\PThursM^{x;\otimes 2}((\pml_{g,m}\setminus \Lambda_C)\times\Lambda_C)=0.
\]
Therefore, $\Lambda_C$ is either null or conull with respect to the
Thurston measure.
\end{proof}

\subsection*{(b) $\Rightarrow$ (c)}
Let $R_+$ and $R_-$ be as in the proof of
\Cref{lem:positivity_of_LC_implies_full}. By
\Cref{lem:general_theory1}, they agree modulo a $\nu_G$-null set.
The preceding argument and condition~(b) show that $R_+$ has full
$\nu_G$-measure. Since $R_+$ agrees modulo null sets with the
conservative part of the time-one map of $\Phi_G$, the flow $\Phi_G$ is
conservative.

\subsection*{(c) $\Rightarrow$ (d)}

For brevity, we use $q$ for both a point of $\mathcal{Q}^1_{g,m}$ and
its image $P_G(q)$ in $\mathcal{Q}^1_{g,m}/G$ when no confusion can arise.

We prove ergodicity in two steps. First, the Hopf argument shows that the ratio limit in \eqref{eq:Ratio-limit} is almost everywhere constant for every compactly supported continuous function $f$ on $\mathcal{Q}^1_{g,m}/G$. We then use conditional expectations to deduce that every invariant measurable set is null or conull. For the Hopf argument, see
\cite[Main Lemma]{Hopf1971}; see also
\cite[Theorem~14]{Sullivan1979} and
\cite[Lemma~3.1]{Hamenstadt2009}.

Fix $x_0\in\teich_{g,m}$. For $q\in\mathcal{Q}^1_{g,m}$, set
\[
r_G(q)=\inf_{[\omega]\in G}d_T(\mathfrak{p}(q),[\omega](x_0)).
\]
Let
\[
\widetilde{\mathcal{A}}(R)
=\{q\in\mathcal{Q}^1_{g,m}\mid r_G(q)\le R\}
\quad\text{and}\quad
\mathcal{A}(R)=P_G(\widetilde{\mathcal{A}}(R))
\subset\mathcal{Q}^1_{g,m}/G.
\]
Then $\mathcal{A}(R)$ is compact. If
$q\in\mathcal{Q}^1_{g,m}$ and $d_T(x,\mathfrak{p}(q))\le R$, then
\[
\ext_x(\lambda_h(q))\le e^{2R}
\quad\text{and}\quad
\ext_x(\lambda_v(q))\le e^{2R}.
\]
In the local coordinates defined in
\eqref{eq:product-coordinatesU1U2-I}, the time coordinate varies over an
interval of length at most $2R$. Therefore, the measure $\nu_G$ satisfies
\begin{equation}
\label{eq:measure_nu2}
\nu_G(\mathcal{A}(R))\le 2Re^{4\convgenus R}
\end{equation}
for $R>0$.
For $a>4\convgenus$ and $b>0$, set
\[
\rho_G(q)=be^{-ar_G(q)}.
\]
Since $\rho_G$ is $G$-invariant, it descends to a function on
$\mathcal{Q}^1_{g,m}/G$. By definition,
\begin{equation}
\label{eq:log-rhoG_asymptotic}
|\log \rho_G(q_1)-\log \rho_G(q_2)|\le ad_T(\mathfrak{p}(q_1),\mathfrak{p}(q_2))
\end{equation}
for $q_1,q_2\in\mathcal{Q}^1_{g,m}$.

The function $\rho_G$ is integrable. Indeed, let
$A_1=\mathcal{A}(1)$ and
$A_k=\mathcal{A}(k)\setminus\mathcal{A}(k-1)$ for $k\ge2$.
It follows from \eqref{eq:measure_nu2} that
\begin{align*}
\int_{\mathcal{A}(N)}\rho_G(q)d\nu_G(q)
&\le b\nu_G(A_1)
+b\sum_{k=2}^Ne^{-a(k-1)}\nu_G(A_k) \\
&\le b\nu_G(A_1)
+2b\sum_{k=2}^Nke^{-a(k-1)+4\convgenus k} \\
&\le b\nu_G(A_1)
+2b\sum_{k=2}^\infty ke^{-a(k-1)+4\convgenus k}<\infty.
\end{align*}
Hence $\rho_G\in L^1(\mathcal{Q}^1_{g,m}/G,\nu_G)$. Choose $b>0$
so that
\[
\int_{\mathcal{Q}^1_{g,m}/G}\rho_G(q)d\nu_G(q)=1.
\]
Conservativity implies that, for almost every
$q\in\mathcal{Q}^1_{g,m}/G$,
\begin{equation}
\label{eq:time-average-rhoG}
\lim_{T\to \infty}\int_0^T\rho_G(\Phi_G^tq)dt=\infty.
\end{equation}
Since $\rho_G>0$, the limit in
\eqref{eq:time-average-rhoG} always exists, possibly with value $\infty$.
Since the time-one map of $\Phi_G$ is conservative and $\rho_G$ is
strictly positive and integrable, the standard recurrence theorem for
conservative transformations gives
\[
\sum_{n=0}^{\infty}\rho_G(\Phi_G^n(q))=\infty
\]
for almost every $q$ (cf.\ \cite[Theorem~3.1]{Krengel1985}). Moreover, \eqref{eq:log-rhoG_asymptotic} implies
that
\[
\rho_G(\Phi_G^t(q))\ge e^{-a}\rho_G(\Phi_G^n(q))
\qquad(n\le t\le n+1).
\]
Integrating this inequality over each interval $[n,n+1]$ proves
\eqref{eq:time-average-rhoG}.

Since $\Phi_G$ is conservative, the Hopf ratio ergodic theorem implies
that, for every compactly supported continuous function $f$ on
$\mathcal{Q}^1_{g,m}/G$, the two limits
\begin{equation}
\label{eq:Ratio-limit}
f_\rho(q)=\lim_{T\to \infty}\frac{\int_0^Tf(\Phi^t_G(q))dt}{\int_0^T\rho_G(\Phi^t_G(q))dt}=\lim_{T\to -\infty}\frac{\int_0^Tf(\Phi^t_G(q))dt}{\int_0^T\rho_G(\Phi^t_G(q))dt}
\end{equation}
exist and agree for almost every $q\in\mathcal{Q}^1_{g,m}/G$
(cf.\ \cite[\S3, pp.~23--24]{Roblin2003}). The function $f_\rho$ is
$\Phi_G$-invariant. Since $f$ has compact support, it is uniformly
continuous, and $\rho_G$ is bounded below on its support. Hence, for
every $\epsilon>0$,
\begin{equation}
\label{eq:f_asymptotic}
|f(\Phi^t_G(q))-f(\Phi^t_G(q'))|\le \epsilon \rho_G(\Phi^t_G(q))
\end{equation}
for all sufficiently large $t>0$ whenever the orbits through $q$ and
$q'$ are forward asymptotic.

Suppose that the ratio limits exist at $q,q'\in\mathcal{Q}^1_{g,m}/G$ and that these points belong to the same strong stable manifold; equivalently,
$\lambda_v(q)=\lambda_v(q')$, with their projective classes in
$\pml_{g,m}^{ue}$. The corresponding orbits are forward asymptotic
\cite{Masur1980}. Applying the same argument to strong unstable
manifolds and using \eqref{eq:log-rhoG_asymptotic} and
\eqref{eq:f_asymptotic}, we conclude that $f_\rho$ is constant along
almost every strong stable and strong unstable manifold (cf. \cite{Hamenstadt2009}, \cite{LindenstraussMirzakhani2008}). Equivalently,
$f_\rho$ induces a measurable function on
$\pml_{g,m}^{ue}\times\pml_{g,m}^{ue}$ that is constant on almost every
horizontal and vertical section. By Fubini's theorem, $f_\rho$ is
constant almost everywhere on $\mathcal{Q}^1_{g,m}/G$.

We now pass from constant ratio limits to ergodicity. Let $\mathcal{B}_G$ be the $\sigma$-algebra of $\Phi_G$-invariant
measurable subsets of $\mathcal{Q}^1_{g,m}/G$. Define a probability
measure $\mu$ by $d\mu=\rho_G\,d\nu_G$.
Consider the conditional expectation
\[
Pf=\mathbb{E}_\mu(f/\rho_G\mid \mathcal{B}_G)
\]
for compactly supported continuous functions $f$ on
$\mathcal{Q}^1_{g,m}/G$. Since the space of such functions is dense in
$L^1(\mathcal{Q}^1_{g,m}/G,\nu_G)$ and
\[
\|Pf-Pg\|_{L^1(\mu)}\le \|f-g\|_{L^1(\nu_G)},
\]
the operator $P$ extends to $L^1(\mathcal{Q}^1_{g,m}/G,\nu_G)$.
The Hopf ratio ergodic theorem gives $Pf=f_\rho$ for every compactly
supported continuous function $f$. Thus $Pf$ is constant almost
everywhere for every $f\in L^1(\mathcal{Q}^1_{g,m}/G,\nu_G)$.

Let $E\in \mathcal{B}_G$. Consider $f_E=\rho_G\boldsymbol{1}_E$. Then
\[
\int_{\mathcal{Q}^1_{g,m}/G}f_E(q)d\nu_G(q)
=\int_{\mathcal{Q}^1_{g,m}/G}\rho_G(q)\boldsymbol{1}_E(q)d\nu_G(q)\le \int_{\mathcal{Q}^1_{g,m}/G}\rho_G(q)d\nu_G(q)=1.
\]
Hence $f_E\in L^1(\mathcal{Q}^1_{g,m}/G,\nu_G)$, so $Pf_E$ is constant
almost everywhere. On the other hand, since $E\in\mathcal{B}_G$,
\[
Pf_E=\mathbb{E}_\mu(f_E/\rho_G\mid \mathcal{B}_G)=\mathbb{E}_{\mu}(\boldsymbol{1}_E\mid \mathcal{B}_G)=\boldsymbol{1}_E.
\]
Thus $\boldsymbol{1}_E$ is constant $\mu$-almost everywhere. Since $\rho_G>0$, the measures $\mu$ and $\nu_G$ are equivalent, so $E$ is either $\nu_G$-null or $\nu_G$-conull. Therefore, $\Phi_G$ is ergodic.

\subsection*{(d) $\Rightarrow$ (e)}
Let $E\subset\pml_{g,m}^2$ be a measurable set invariant under the
diagonal action of $G$. Then
\[
\tilde{E}=\{q\in \mathcal{Q}^1_{g,m}\mid ([\lambda_h(q)],[\lambda_v(q)])\in E\}
\]
is both $G$-invariant and $\Phi^t$-invariant. It therefore descends to a
$\Phi_G$-invariant set in $\mathcal{Q}^1_{g,m}/G$. Since $\Phi_G$ is
ergodic, $\tilde{E}$ is either null or conull in
$\mathcal{Q}^1_{g,m}$. Since $\Delta'$ is null, the local product
structure of $\nu$ and the mutual absolute continuity of $\PSmeasure$
and $\PThursM^{x;\otimes2}$ imply that $E$ is correspondingly either
null or conull in $\pml_{g,m}^2$.
Thus the diagonal action of $G$ is ergodic.

\subsection*{(e) $\Rightarrow$ (a)}
We argue by contradiction. Assume that the diagonal action of $G$ is ergodic but that $\PThursM^x(\Lambda_C(G))=0$. We will construct a measurable wandering fundamental domain on a conull subset of $\pml_{g,m}^2$, making the diagonal action totally dissipative.
Let $\boldsymbol{X}_0$ denote the complement of the diagonal in $(\pml_{g,m}^{ue})^2$, and define
\begin{equation}
\label{eq:definition_X}
\boldsymbol{X}=\{([\lambda_1],[\lambda_2])\in \boldsymbol{X}_0\mid [\lambda_1],[\lambda_2]\notin \Lambda_C(G)\}\setminus \bigcup_{[\omega]\in G\setminus \{\mathrm{id}\}}\operatorname{Fix}([\omega])^2,
\end{equation}
where $\operatorname{Fix}([\omega])=\{[\lambda]\in \pml_{g,m}\mid [\omega]([\lambda])=[\lambda]\}$.
The set $\boldsymbol{X}$ is invariant under the diagonal action of $G$, and every point of $\boldsymbol{X}$ has trivial stabilizer.

Let $\boldsymbol{\lambda}=([\lambda_1],[\lambda_2])\in \boldsymbol{X}$. Since $[\lambda_1]\ne[\lambda_2]$ and both laminations are uniquely ergodic, $\lambda_1$ and $\lambda_2$ bind $\Sigma_{g,m}$. That is, $i(\lambda_1,\eta)+i(\lambda_2,\eta)>0$ for all $\eta\in \ml_{g,m}\setminus \{0\}$. Therefore, there is a holomorphic quadratic differential $q_{\boldsymbol{\lambda}}$ whose horizontal and vertical laminations represent the projective classes of $\lambda_1$ and $\lambda_2$, respectively.
Let $L_{\boldsymbol{\lambda}}\colon \mathbb{R}\to \teich_{g,m}$ be the Teichm\"uller geodesic determined by $q_{\boldsymbol{\lambda}}$. The pair $\boldsymbol{\lambda}$ determines this geodesic up to translation of its parameter; in particular, its image depends only on $\boldsymbol{\lambda}$. We denote this image by $\operatorname{Im}L_{\boldsymbol{\lambda}}$.

Fix $x_0\in \teich_{g,m}$. 
We claim that for any $R>0$, 
\begin{equation}
    \label{eq:counting_not_conical_set}
{}^\#\{[\omega]\in G\mid d_T([\omega](x_0),\operatorname{Im}L_{\boldsymbol{\lambda}})\le R\}<\infty.
\end{equation}
Indeed, suppose that this set is infinite for some $R>0$. Since the action of $G$ on $\teich_{g,m}$ is properly discontinuous, there is a sequence $\{[\omega_n]\}_n\subset G$ of distinct elements such that $d_T([\omega_n](x_0),\operatorname{Im}L_{\boldsymbol{\lambda}})\le R$ and $[\omega_n](x_0)$ leaves every compact subset of $\teich_{g,m}$. Choose $t_n\in\mathbb{R}$ such that
\[
d_T([\omega_n](x_0),L_{\boldsymbol{\lambda}}(t_n))\le R+1.
\]
After passing to a subsequence, we have either $t_n\to+\infty$ or $t_n\to-\infty$. It follows from the definition of a conical limit point that one of $[\lambda_1]$ and $[\lambda_2]$ belongs to $\Lambda_C(G)$, contradicting the definition of $\boldsymbol{X}$.

We define
\begin{equation}
\label{eq:inf-geodesic_line}
\boldsymbol{\Delta}(\boldsymbol{\lambda})=\inf_{[\omega]\in G}d_T([\omega](x_0),\operatorname{Im}L_{\boldsymbol{\lambda}}).
\end{equation}
By \eqref{eq:counting_not_conical_set}, the infimum in \eqref{eq:inf-geodesic_line} is attained for every $\boldsymbol{\lambda}\in \boldsymbol{X}$, and
\[
G(\boldsymbol{\lambda})=\{[\omega]\in G\mid d_T([\omega](x_0),\operatorname{Im}L_{\boldsymbol{\lambda}})=\boldsymbol{\Delta}(\boldsymbol{\lambda})\}
\]
is a nonempty finite subset of $G$.
For $[\omega]\in G$, set
\begin{equation*}
\operatorname{Min}([\omega])
=\{\boldsymbol{\lambda}\in \boldsymbol{X}\mid d_T([\omega](x_0),\operatorname{Im}L_{\boldsymbol{\lambda}})=\boldsymbol{\Delta}(\boldsymbol{\lambda})\}.
\end{equation*}
For each nonempty finite subset $E\subset G$, set
\begin{align}
\boldsymbol{X}(E)
&=\{\boldsymbol{\lambda}\in \boldsymbol{X}\mid G(\boldsymbol{\lambda})=E\}
=
\left(\bigcap_{[\omega]\in E}\operatorname{Min}([\omega])\right)\setminus
\left(\bigcup_{[\omega]\in G\setminus E}\operatorname{Min}([\omega])\right).
\label{eq:inf-geodesic_line2}
\end{align}
Some of the sets $\boldsymbol{X}(E)$ may be empty.
By \eqref{eq:counting_not_conical_set}, every
$\boldsymbol{\lambda}\in\boldsymbol{X}$ belongs to
$\boldsymbol{X}(E)$ for a unique nonempty finite subset $E\subset G$.
In particular, the sets $\boldsymbol{X}(E)$ are pairwise disjoint, and hence
\begin{equation}
\label{eq:equivariant-decomposition}
\boldsymbol{X}
=\bigsqcup_{\substack{E\subset G\\0<\#E<\infty}}\boldsymbol{X}(E).
\end{equation}
The following claim gives the measurability and equivariance needed to construct a wandering domain.

\begin{claim}
\label{claim:1_XE_is_Borel}
\begin{itemize}
\item[{\rm (A)}]
For every nonempty finite subset $E\subset G$,
$\boldsymbol{X}(E)$ is Borel in $\boldsymbol{X}$. 
\item[{\rm (B)}] The decomposition \eqref{eq:equivariant-decomposition} is equivariant with respect to the diagonal action of $G$ in the following sense:
\begin{itemize}
\item[(i)]
$\boldsymbol{X}([\omega]E)=[\omega](\boldsymbol{X}(E))$ for $[\omega]\in G$, where $[\omega]E=\{[\omega]\circ [\eta]\mid [\eta]\in E\}$; and
\item[(ii)]
if $[\omega](\boldsymbol{X}(E))\cap \boldsymbol{X}(E)\ne \emptyset$, then $[\omega]E=E$.
\end{itemize}
\end{itemize}
\end{claim}
\begin{proof}[Proof of \Cref{claim:1_XE_is_Borel}]
\textrm{(A)}\ It suffices to show that $\operatorname{Min}([\omega])$ is Borel for every $[\omega]\in G$.
To this end, we identify $\boldsymbol{X}$ with the set
\[
\{(\lambda_1,\lambda_2)\in \ml_{g,m}^2\mid ([\lambda_1],[\lambda_2])\in \boldsymbol{X}, \ext_{x_0}(\lambda_1)=\ext_{x_0}(\lambda_2)=1\}.
\]
For $\boldsymbol{\lambda}=(\lambda_1,\lambda_2)\in \boldsymbol{X}$, let $Q_{\boldsymbol{\lambda}}$ be the holomorphic quadratic differential satisfying $\HMh(Q_{\boldsymbol{\lambda}})=\lambda_1$ and $v(Q_{\boldsymbol{\lambda}})=\lambda_2$.
The map
\[
\boldsymbol{X}\ni \boldsymbol{\lambda}\mapsto Q_{\boldsymbol{\lambda}}\in \mathcal{Q}_{g,m}
\]
is continuous with respect to the topology induced by $\ml_{g,m}^2$. Parametrize $L_{\boldsymbol{\lambda}}$ so that $L_{\boldsymbol{\lambda}}(0)$ is the underlying surface of $Q_{\boldsymbol{\lambda}}$. For every fixed $t\in\mathbb R$ and $[\omega]\in G$, the function
\[
\boldsymbol{\lambda}\longmapsto d_T([\omega](x_0),L_{\boldsymbol{\lambda}}(t))
\]
is continuous. Continuity in $t$ also gives
\[
d_T([\omega](x_0),\operatorname{Im}L_{\boldsymbol{\lambda}})
=\inf_{t\in\mathbb Q}d_T([\omega](x_0),L_{\boldsymbol{\lambda}}(t)).
\]
Thus the distance to the geodesic is upper semicontinuous and Borel measurable, which is all that is needed here. Taking the further infimum over the countable group $G$ shows that $\boldsymbol{\Delta}$ is Borel measurable. Hence $\operatorname{Min}([\omega])$, defined by equality of two Borel functions, is Borel in $\boldsymbol{X}$.

\medskip
\noindent
\textrm{(B)}\
For $[\omega]\in G$ and $\boldsymbol{\lambda}\in\boldsymbol{X}$, equivariance of Teichm\"uller geodesics gives
\[
\operatorname{Im}L_{[\omega](\boldsymbol{\lambda})}
=[\omega](\operatorname{Im}L_{\boldsymbol{\lambda}}).
\]
Since the action is isometric, we obtain
\[
G([\omega](\boldsymbol{\lambda}))=[\omega]G(\boldsymbol{\lambda}).
\]
This proves~(i). If $[\omega](\boldsymbol{X}(E))\cap\boldsymbol{X}(E)\ne\emptyset$, a point in this intersection has minimizing set both $[\omega]E$ and $E$. Hence $[\omega]E=E$, proving~(ii).
\end{proof}

We now use the claim to select one representative from each orbit in $\boldsymbol{X}$.
Since the Thurston measure is nonatomic, the Rokhlin isomorphism theorem provides a Borel isomorphism (cf. \cite[Theorem~17.41]{Kechris1995})
\[
f\colon (\pml_{g,m}^2,\PThursM^{x_0,\otimes 2})\longrightarrow ([0,1],\text{Lebesgue measure}).
\]

For a nonempty finite subset $E\subset G$, set
\[
\operatorname{Stab}_G(E)=\{[\omega]\in G\mid [\omega]E=E\}.
\]
This subgroup is finite because its action on the finite set $E$ by left multiplication is faithful. By \eqref{eq:definition_X}, its action on $\boldsymbol{X}(E)$ is free. Selecting the point with the smallest $f$-value in each stabilizer orbit, we obtain the set
\[
\boldsymbol{X}_F(E)=\{\boldsymbol{\lambda}\in \boldsymbol{X}(E)\mid \text{$f(\boldsymbol{\lambda})<f([\omega](\boldsymbol{\lambda}))$ for $[\omega]\in \operatorname{Stab}_G(E)\setminus \{\mathrm{id}\}$}\},
\]
which is a Borel fundamental domain for the action of $\operatorname{Stab}_G(E)$ on $\boldsymbol{X}(E)$.

Let $\mathscr{R}$ be a set of representatives for the $G$-orbits of nonempty finite subsets of $G$ under left multiplication, and set
\[
W=\bigcup_{E\in\mathscr{R}}\boldsymbol{X}_F(E).
\]
The set $W$ is Borel because $G$, and hence $\mathscr{R}$, is countable. We claim that $W$ is wandering for the action of $G$. Indeed, if
\[
[\omega](\boldsymbol{X}_F(E))\cap\boldsymbol{X}_F(E')\ne\emptyset
\]
for some $E,E'\in\mathscr{R}$, then \Cref{claim:1_XE_is_Borel}(B)(i) gives $[\omega]E=E'$. The choice of $\mathscr{R}$ implies that $E=E'$, and hence $[\omega]\in\operatorname{Stab}_G(E)$. Since $\boldsymbol{X}_F(E)$ is a fundamental domain for the action of this stabilizer, we must have $[\omega]=\mathrm{id}$. Thus, $W$ is wandering. The equivariant decomposition \eqref{eq:equivariant-decomposition} and the definition of $\boldsymbol{X}_F(E)$ now give
\[
\boldsymbol{X}=\bigcup_{[\omega]\in G}[\omega](W).
\]

Finally, $\pml_{g,m}^{ue}$ has full Thurston measure, and the diagonal in $(\pml_{g,m}^{ue})^2$ is null because the Thurston measure is nonatomic. By hypothesis, $\Lambda_C(G)$ is null. Furthermore, $G$ is countable, and $\operatorname{Fix}([\omega])^2$ is null with respect to $\PThursM^{x_0;\otimes 2}$ for every $[\omega]\in G\setminus\{\mathrm{id}\}$ (cf. \cite{miyachi2025functiontheorydynamicsergodic}). It follows from \eqref{eq:definition_X} that $\boldsymbol{X}$ has full measure in $\pml_{g,m}^2$. Since $\boldsymbol{X}$ is the union of the $G$-translates of the wandering set $W$, the diagonal action of $G$ is totally dissipative. Moreover, $W$ has positive measure because $G$ is countable and its action is nonsingular. Since $\PThursM^{x_0;\otimes2}$ is nonatomic, we can split $W$ into two measurable sets of positive measure. Their $G$-saturations are disjoint invariant sets of positive measure, contradicting ergodicity. Thus $\PThursM^x(\Lambda_C(G))>0$, completing the proof.

% \bibliographystyle{plain}
% \bibliography{Hopf_Subgroup_MCG_revise.bib}

\def\cprime{$'$} \def\cprime{$'$}

\end{document}